\documentclass{amsart}

\usepackage{amsfonts}
\usepackage{amssymb}
\usepackage{amscd}

\theoremstyle{plain} 
\newtheorem{thm}{Theorem}[section]
\newtheorem{lem}[thm]{Lemma}
\newtheorem{prop}[thm]{Proposition}
\newtheorem{cor}[thm]{Corollary}
\newtheorem{conj}[thm]{Conjecture}
\theoremstyle{definition}
\newtheorem{dfn}[thm]{Definition}
\newtheorem{rmk}[thm]{Remark}
\newtheorem{ex}[thm]{Example}

\begin{document}
\title[Near-Riesz and pseudo-Riesz bases, C*-Fredholm operators]{Near-Riesz and pseudo-Riesz bases of Hilbert C*-modules, and semi-C*-Fredholm operators}
\author{Michael Frank}
\address[M. Frank]{HTWK Leipzig\\Fakult\"at Informatik und Medien\\PF 301166\\04251 Leipzig\\Germany}
\email[M. Frank]{michael.frank@htwk-leipzig.de, michael.frank.leipzig@gmx.de}
\author{Stefan Ivkovi\'c}
\address[S. Ivkovi\'c]{Mathematical Institute\\Serbian Academy of Sciences and Arts\\Department of Mathematics, Physics and Geo Sciences\\Kneza Mihaila 36, Belgrade 11001\\Serbia}
\email[S. Ivkovic]{stefan.iv10@outlook.com}
\date{November 3rd, 2026, version 5.3}

\begin{abstract}
Several useful concepts of almost frames or Riesz bases of Hilbert spaces gave us the motivation to review the theory of C*-module frames and C*-module Riesz bases for Hilbert C*-modules, searching for similar useful structures. We give a survey of the background in C*-module frame theory, summarizing and particularly extending or correcting earlier results. The concepts of near-Riesz bases, pseudo-frames, pseudo-Riesz sequences, and pseudo-Riesz bases of Hilbert C*-modules are introduced and equivalently related to (semi-)C*-Fredholm properties of their synthesis operators. The notion of the excess of a C*-module frame is introduced, taking values in the related C*-algebras. Indicators for possible erasures of frame elements are described with the aim to preserve robustness of frames. Introducing a new notion of pseudo-Riesz-Weyl sequences in the standard Hilbert C*-module $l_2(A)$ we reverse the found correspondence by assuming the associated synthesis operator of a Bessel sequence to be (upper semi-)C*-Weyl. Furthermore, we deduce from the invariance of (upper semi-)C*-Fredholm and C*-Riesz-Weyl properties of synthesis operators with respect to ``compact'' perturbations the perturbation invariance of the respective classes of module Bessel sequences. 
\end{abstract}

\maketitle


\section{Introduction}
The aim of the present investigations is the exploration of ways to types of Bessel sequences of algebraically and countably generated Hilbert C*-modules that preserve certain reconstruction properties like (modular) frames or Riesz bases approximately. We follow ideas by J.~A.~Holub, by P.~Casazza and O.~Christensen, and by D.~Biswas and M.~Mitkowski and by C.~Heil and Pu-Ting Yu for Hilbert spaces (\cite{Holub_1994,CC,BM_2025,HY_2023}) finding analogous structures in a suitable Hilbert C*-modules over unital C*-algebras $A$. First, we go the way of adding and/or deleting a finite set of elements of C*-module Bessel sequences belonging to certain algebraically generated Hilbert C*-modules to and/or from the admissible classes of C*-module Bessel sequences. Here notions like erasures of frames and robustness of frames play a role. Secondly, we discover alternative characteristics of synthesis operators by lower semi-$A$-Fredholm, upper semi-$A$-Fredholm or $A$-Fredholm operators expressing relevant properties of the selected classes in terms of operator properties without reference to concrete addition or elimination processes in C*-module Bessel sequences, cf.~\cite{Ivkovic_2021,Ivkovic_2019,Ivkovic_2020_1,Ivkovic_2020_2,Ivkovic_2023,Ivkovic_2024_2,Ivkovic_2024_1}. By the way, we extend the notion of the excess of a frame from Hilbert spaces to the context of C*-linear frames of Hilbert C*-modules finding another frame invariant beyond the $K_0(A)$-class of related algebraically finitely generated Hilbert C*-submodules, cf.~\cite{CC,ACP,BB,Heil}. However, it takes not necessarily values in the non-negative integers anymore, but in the C*-algebra of coefficients or in its bidual W*-algebra. Reversing the direction of the investigation, we consider Bessel sequences in Hilbert C*-modules the synthesis operators of which are upper semi-C*-Weyl or C*-Weyl operators, i.e. of subclasses of the respective sets of (semi-)C*-Fredholm operators. Perturbation results for the four named classes of module Bessel sequences that can be derived from perturbation invariances of their respective synthesis operators are another useful application of the found correspondence. 

At the beginning, we start with reviewing, partially completing and summarizing important results of module frame theory of Hilbert C*-modules. Then we introduce two types of C*-module Bessel sequences that are not necessarily C*-module frames or C*-module Riesz bases, but still possess reasonably good properties for calculations as with C*-module frames and C*-module Riesz bases. We generalize the concepts of near-Riesz and pseudo-Riesz bases for Hilbert spaces in the context of Hilbert C*-modules (cf.~Definitions \ref{def_near_Riesz_basis}, \ref{def_pseudo_frame}, \ref{def_pseudo} and \ref{def_pseudo_Riesz_basis}) and establish equivalent characterizations in terms of (semi-)C*-Fredholm operators on Hilbert C*-modules (cf. Cor.~\ref{cor_A-Fredholm}, Thms.~\ref{prop_pseudo1}, \ref{prop_pseudo2}, \ref{prop_pseudo3}). Furthermore, generalizing semi-Weyl operators on Hilbert spaces we introduce pseudo-Riesz-Weyl and sequences and bases by the (semi-)C*-Weyl properties of their synthesis operators, Props.~\ref{semi-weyl}, \ref{weyl}. Applying the found correspondence between classes of synthesis operators and classes of module Bessel sequences, we establish some perturbation results of these classes of Bessel sequences, Thm.~\ref{perturbacija}, Props.~\ref{prop_perturb1}, \ref{prop_perturb2}, \ref{prop_perturb3}.

We rely on earlier work by the first author and D.~R.~Larson \cite{FL_2000_1,FL_2000_2,FL_2002}, by Wu Jing \cite{Jing_2006}, by Deguang Han, Wu Jing, D.~R.~Larson, and R.~N.~Mohapatra \cite{HJLM_2008,HJM_2009}, by  A.~Khosravi, B.~Khosravi \cite{KhKh}, by M.~Hasannasab \cite{Hasannasab}, and by others (see the revisiting paper \cite{XCL-2024}) on the C*-module frame theory for Hilbert C*-modules. The applied research results on 
(semi-)$A$-Fredholm operator theory have their origins in publications by A.~S.~Mi\-shchen\-ko and A.~T.~Fomenko \cite{MF}, by V.~M.~Manuilov and E.~V.~Tro{\"{\i}}tsky \cite{MT_2023}, and by the second author \cite{Ivkovic_2023,Ivkovic_2019,Ivkovic_2020_1,Ivkovic_2020_2,Ivkovic_2024_2,Ivkovic_2024_1}, among others.

\section{Preliminaries}

In this section we collect, summarize and partially complete some facts from Hilbert C*-module theory used in the forthcoming investigations. 

The following fact has been known as the Serre-Swan theorem in the commutative case, and as a Dixmier-Douady theorem in the non-commutative case. We restate it in full complexity for convenience, cf.~\cite[Thm.~15.4.2, Rem.~15.4.3, Cor.~15.4.8]{NEWO}, \cite[Thm.~5.9]{FL_2002}.

\begin{lem} \label{SS-DD}
Let $A$ be a unital C*-algebra. The following conditions on a left $A$-module $\mathcal M$ are equivalent:
\begin{itemize}
   \item[(i)] $\mathcal M$ is projective and algebraically finitely generated.
   \item[(ii)] $\mathcal M$ is isomorphic to a direct summand in $A^n$ for some $n \in \mathbb N$ as a module.
   \item[(iii)] $\mathcal M$ is isomorphic to an algebraically finitely generated submodule and direct summand of $l_2(A)$ as a module.
   \item[(iv)] $\mathcal M$ is the range of a ``compact'' orthogonal projection on $l_2(A)$, considering the latter as the standard countably generated Hilbert $A$-module.
\end{itemize}
In particular, every algebraically finitely generated projective $A$-module can be endowed with a structure of a self-dual Hilbert $A$-module. Moreover,
Hilbert $A$-modules are projective and algebraically finitely generated as $A$-modules if and only if any adjointable module operator on them is ``compact''. Finally, every algebraically finitely generated Hilbert $A$-module is projective. Any finite set of algebraic generators is a module frame.
\end{lem}

Note, that there is a difference between algebraically and topologically finitely generated Hilbert C*-modules, subsuming the latter among the countably generated ones. As an example (\cite{FL_2002}), consider the C*-algebra of all continuous functions $A=C([0,1])$ on the unit interval and its ideal of all continuous functions equal to zero at $0 \in [0,1]$, $C_0([0,1])$. By the (Stone-)Weierstrass theorem all polynomials on $[0,1]$ with zero absolut summand are norm-dense in $C_0([0,1])$, and so the element $f(t)=t$, $t \in [0,1]$, generates $C_0([0,1])$ as a Hilbert $A$-module topologically, but not algebraically.

Considering a dimension notion of algebraically finitely generated Hilbert $A$-modules over certain unital C*-algebras $A$, one realizes soon that the monoid of all algebraically finitely or countably generated Hilbert $A$-modules may lack the cancellation property and/or a linearly ordered dimension function. This can be seen constructing the respective $K_0$-group of $A$: for example, for the strong Morita equivalence class $\{ {\mathbb C}; \, M_n({\mathbb C}) : \, n \in {\mathbb N};  \, K(H): \, H \, - \, {\rm infinite}{-}{\rm dimensional} \,\, {\rm Hilbert} \,\, {\rm space} \}$ the $K_0$-group is nice and equal to $\mathbb Z$. For compact C*-algebras (i.e. those admitting a faithful $*$-representation in a certain C*-algebra $K(H)$) the group $K_0(A)$ admits at least a partial ordering and a partial cancellation. But, for more general C*-algebras any kind of integer-valued dimension function might fail. So, the notion of the excess of a frame, as defined for Hilbert space frames, may get at least a non-numerical sense for general Hilbert C*-modules. We will give a definition of the excess of a modular frame in Section 3, Definition \ref{def_excess}.

Sometimes algebraically finitely generated Hilbert C*-modules might be also countably algebraically generated. The coincidence of these two types of algebraic generatedness was demonstrated by D.~P.~Blecher and T.~Kania in \cite[Cor.~3.8]{BK_2014}. A counterexample shows extra phenomena for certain uncountably algebraically generated Hilbert C*-modules.

One of the milestones in Hilbert C*-module theory is Kasparov's stabilization theorem for certain countably generated Hilbert C*-modules. It was improved for the context of module frame theory in \cite{FL_2002},  cf.~\cite{Kasp}, \cite[Thm.~15.4.6, Cor.~15.4.5, Lem.~15.4.4, Exerc.~15.P]{NEWO}, \cite[Thm.~4.1]{FL_2002}).

\begin{lem} \label{K}
Let $A$ be a unital C*-algebra and $\mathcal M$ be a finitely or countably Hilbert $A$-module. Then  $\mathcal M$ is unitarily isomorphic to a direct orthogonal summand of the standard countably generated Hilbert $A$-module $l_2(A)$ with an orthogonal complement unitarily isomorphic to $l_2(A)$ again, i.e. $l_2(A) = {\mathcal M} \oplus l_2(A)$. 
The following conditions are further equivalent conditions for unital C*-algebras:
\begin{itemize}
   \item[(i)]  $\mathcal M$ is finitely or countably generated as a Hilbert $A$-module.
   \item[(ii)] $\mathcal M$ admits a finite or countable standard modular frame as a Hilbert $A$-module (see Definition \ref{frame-def} below).
   \item[(iii)] The C*-algebra ${\rm K}_A({\mathcal M})$ of all ``compact'' operators on $\mathcal M$ is $\sigma$-unital, i.e. it admits a strictly positive element. (Equivalently, it admits a positive operator $T$ on $\mathcal M$ with norm-dense image in $\mathcal M$, i.e. such that $T \cdot {\rm K}_A({\mathcal M})$ is norm-dense in ${\rm K}_A({\mathcal M})$.)
\end{itemize}
\end{lem}   

The assumption on $A$ to be unital can be weakened to $A$ being $\sigma$-unital for the first assertion, because both $A^n$, $n \in \mathbb N$, and $l_2(A)$ are still countably generated precisely in these cases.  In \cite{RT_2002} multiplier-module-valued modular frames have been used to extend the notion of countably generatedness of Hilbert C*-modules to the general non-unital case, and to demonstrate the Kasparov stabilization theorem to be valid for arbitrary C*-algebras $A$ in that extended sense, i.e. the first assertion.


\begin{dfn}  \label{frame-def} ({\rm cf.~\cite{FL_2000_1,FL_2000_2,FL_2002}, \cite{HJLM_2008}, \cite[Def.~1.2]{Hasannasab}}) \newline
Let $A$ be a unital C*-algebra and $I$ be a finite or countable index subset of $\mathbb N$. A sequence $\{ e_i \}_{i \in I}$ of non-zero elements in a Hilbert $A$-module $\mathcal X$ is said to be an {\it orthogonal Hilbert basis} if $\langle e_i,e_j \rangle = 0$ for any $i \not= j$ and all elements $\langle e_i,e_i \rangle = \langle e_i,e_i \rangle^2$ are orthogonal projections of $A$. If additionally  $\langle e_i,e_i \rangle = 1_A$ for any $i \in I$ then $\{ x_i \}_{i \in I}$ is said to be an {\it orthonormal (Hilbert) basis of $\mathcal X$}.

A sequence $\{ x_i \}_{i \in I}$ of elements in a Hilbert $A$-module $\mathcal X$ is said to be a {\it Bessel sequence} if there exists a real positive constant $M$ such that
\begin{equation} \label{Bessel}
    \sum_{i \in I} \langle x,x_i \rangle \langle x_i,x \rangle \leq M \cdot \langle x,x \rangle
\end{equation}
for every $x \in \mathcal X$. A sequence  $\{ x_i \}_{i \in I}$ of elements in a Hilbert $A$-module $\mathcal X$ is said to be a {\it frame} if there are real positive constants $m \leq M$ such that
  \begin{equation}\label{frame}
    m \cdot \langle x,x \rangle \leq \sum_{i \in I} \langle x,x_i \rangle \langle x_i,x \rangle \leq M \cdot \langle x,x \rangle
\end{equation}  
for every $x \in \mathcal X$. The optimal constants $m, M$ are called {\it frame bounds}. A frame is said to be {\it tight} if $m=M$, and it is said to be {\it normalized} (or a {\it Parseval frame}) if $m=M=1$. Concerning the type of convergence of the sums in (\ref{Bessel}) and (\ref{frame}) we resort to {\it standard frames}, i.e. to C*-norm-convergence of these sums. For infinite-dimensional C*-algebras there are often other types of convergence of infinite sums like strikt, weak, weak*, in order, for example. The results are often larger sets or spaces as their (norm-closed) $A$-linear hulls.

A sequence  $\{ x_i \}_{i \in I}$ of elements in a Hilbert $A$-module $\mathcal X$ is said to be a {\it Riesz basis of $\mathcal X$} if this sequence is a frame without zero elements and, additionally, finite or infinite sums of type $\sum_{i \in J \subseteq I} a_i x_i$ with $\{ a_i \}_{i \in J \subseteq I} \in A$ (w.r.t. any kind of suitable convergence in $\mathcal X$) equal zero if and only if every single summand $\{ a_ix_i \}_{i \in J \subseteq I}$ equals zero. A Riesz basis is {\it standard} if it is a standard frame. Wu Jing proved that the consideration of sequences $\{ a_i \}_{i \in \mathbb N}$ of $l_2(A)$ is necessary and sufficient to characterize Riesz bases, \cite[Thm.~4.1, Cor.~4.3]{Jing_2006}. 

Let $\mathcal X$ be algebraically finitely or countably generated. A sequence $\{ x_j \}_{i \in I}$ is called a {\it modular Riesz basis for countably generated $\mathcal X$} if the $A$-linear operator of $l_2(A)$ onto $\mathcal X$, that maps every element $e_i$ of the fixed orthonormal Hilbert basis $\{ e_i \}_{i=1}^\infty$ of $l_2(A)$ to the corresponding sequence element $x_i$, is invertible and adjointable. 
A finite sequence $\{ x_j \}_{i =1}^n$ is called a {\it modular Riesz basis for algebraically finitely generated $\mathcal X$} if the $A$-linear operator of $A^n$ onto $\mathcal X$, that maps every element $e_i$ of the fixed orthonormal Hilbert basis $\{ e_i \}_{i=1}^n$ of $A^n$ to the corresponding sequence element $x_i$, is invertible and adjointable. 

If the norm-closed $A$-linear span of the frame elements is only a Hilbert $A$-submodule of $\mathcal X$ we call frame sequences {\it quasi-frames} and indicate their norm-closed $A$-linear span, i.e. the hosting Hilbert $A$-submodule. 
\end{dfn}

Since the non-negative (self-adjoint) elements of a C*-algebra form a cone the convergence of norm-convergent sums of non-negative elements is unconditional.     The norms of the Bessel and frame elements $\{ \|x_i\| \}_{i \in I}$ are uniformly bounded by $\sqrt{M}$.  Remarkably, there are countably Hilbert C*-modules with frames, but without any Riesz basis. An example was already given above with $A=C([0,1])$ and the topologically singly generated Hilbert $A$-module ${\mathcal X} = C_0([0,1])$ considered as an isometrically embedded Hilbert $A$-submodule and direct orthogonal summand of $l_2(C([0,1]))$, cf.~\cite{FL_2002}. It does not contain any orthogonal Hilbert basis, in particular.

\begin{dfn}
For a standard Bessel sequence $\{ x_i \}_{i \in I}$ of a Hilbert $A$-module $\mathcal X$ there are two canonical bounded $A$-linear operators, the {\it analysis operator} or {\it frame transform} $\theta: {\mathcal X} \to l_2(A)$, defined by the formula $\theta(x)= \sum_{i \in I} \langle x,x_i \rangle e_i$ for $x \in \mathcal M$ and a orthonormal basis $\{ e_i \}_{i \in \mathbb N}$ of $l_2(A)$, and the {\it synthesis operator} $\sigma: l_2(A) \to \mathcal X$, defined by $\sigma (c) = \sum_{i \in I} c_i x_i=\theta^*(c)$ for all sequences $c= \{ c_i \}_{i \in \mathbb N} \in l_2(A)$ with respect to the same orthonormal basis $\{ e_i \}_{i \in \mathbb N}$. The combined operator $S=\sigma \circ \theta$, $S: {\mathcal X} \to {\mathcal X}$, defined by $S(x) = \sum \langle x,x_i \rangle x_i$ for $x \in \mathcal X$, is said to be the {\it frame operator}.  \newline
A (standard) frame $\{ x_i \}_{i \in I}$ of a Hilbert $A$-module $\mathcal X$ is {\it similar} to another (standard) frame $\{ y_i \}_{i \in I}$ of $\mathcal X$ if there exists a invertible adjointable operator $T$ with the property $x_i = T(y_i)$ for any $i \in I$. Such two frames are {\it unitarily isomorphic} if $T$ is a unitary operator. \newline
A (standard) frame $\{ y_i \}_{i \in I}$ of a Hilbert $A$-module $\mathcal X$ is {\it a dual frame} of a given (standard) frame $\{ x_i \}_{i \in I}$ of $\mathcal X$ if the reconstruction formulae $x = \sum_{i \in I} \langle x,y_i \rangle x_i = \sum_{i \in I} \langle x,x_i \rangle y_i$ are valid for any $x \in \mathcal X$, (where the sums converge in norm in the standard case). 
\end{dfn}

The frame operator $S$ is positive as an element of the C*-algebra ${\rm End}_A^*({\mathcal X})$. For a standard modular Riesz basis the operator $S$ is invertible. For a standard Riesz basis $\theta$ is injective and $\sigma=\theta^*$ is surjective. Note, that the sequence $\{ S^{1/2}(x_i) \}_{i \in I}$ is a standard normalized tight frame in the case of a standard frame $\{ x_i \}_{i \in I}$. 

If the unital C*-algebra $A$ has a sufficiently rich structure then non-frame Bessel sequences of the singly generated Hilbert $A$-module ${\mathcal X}=A^1=A$ can consist of just one element. Indeed, consider C*-algebras like $A=C([0,1])$ or $A=B(l_2)$ and take an element $f$ of them which has zero as a non-isolated point of its spectrum. Then $\{ f \}$ is a Bessel sequence of $\mathcal X$ but it is not a frame of it. So, in contrast to the Hilbert space case a non-frame Bessel sequence has not have to consist of infinitely many elements. Note, that for an algebraically finitely generated Hilbert $A$-module $\mathcal X$ any finite set of generators of ${\mathcal X}$ generating $\mathcal X$ algebraically has to be a frame of $\mathcal X$, \cite[Thm.~5.9]{FL_2002}. 

As another example consider two generating sets of the (left) Hilbert $M_2({\mathbb C})$-module ${\mathcal X}=M_2({\mathbb C})=A^1$, the set $\{ 1_A \}$ and a set of two orthogonal minimal projections $\{ p_1,p_2 \}$. These two sequences are both Riesz bases of $\mathcal X$ with different ranges of the resulting analysis operators in  $l_2(A)$, namely $A$ in $A^1$ and $Ap_1 \oplus Ap_2$ in $A \oplus A = A^2$. So, only the first set is a modular Riesz basis. The two Riesz bases are not unitarily isomorphic nor similar, (cf.~\cite[Ex.~6.4]{FL_2002}, \cite[Ex.~3.4]{HJLM_2008}). The deeper reason for their non-similarity can be found in the interrelation of the ranges of their analysis operators: neither one is contained in the other, cf.~\cite[Lemma 3.9]{HJLM_2008}.
Moreover, e.g.~adding a single element of a Riesz basis to it still gives a frame, but the property to be a Riesz basis gets lost for the newly built sequence. 

For standard frames of a Hilbert $A$-module $\mathcal X$ the operator $\theta$ is adjointable and admits a closed range in $l_2(A)$ which is an orthogonal direct summand unitarily isomorphic to $\mathcal X$. Its adjoint operator is $\sigma$ has also a closed range which is an orthogonal direct summand of $\mathcal X$, cf.~\cite[Thm.~4.1, Thm.~4.4]{FL_2002}. One has $\theta^*(e_i) = \sigma (e_i) = x_i$ for $i \in I$ and $P$ the orthogonal projection of $l_2(A)$ onto the range of $\theta$. If the frame $\{ x_i \}_{i \in I}$ is normalized tight then $P(e_i)= \theta(x_i)$ for $i \in I$, cf.~\cite[Thm.~4.1]{FL_2002}.
In particular, the elements of a standard frame $\{ x_i \}_{i\in I}$ form a generating set of $\mathcal X$ as a Hilbert $A$-module, but not every generating set is a frame. What is more, for every standard frame $\{ x_i \}_{i \in I}$ we have the {\it reconstruction formula} 
\begin{equation*}
     x = \sum_{i \in I} \langle x, (\theta^* \theta)^{-1} (x_i) \rangle x_i = \sum_{i \in I} \langle x,x_i \rangle (\theta^* \theta)^{-1}(x_i)
\end{equation*}
for any $x \in \mathcal X$, cf.~\cite[Thm.~6.1, Thm.~6.2]{FL_2002}. For normalized tight frames $\{ x_i \}_{i \in I}$ the reconstruction formula reduces to
\begin{equation*}
    x = \sum_{i \in I} \langle x, x_i \rangle x_i
\end{equation*}
for any $x \in \mathcal X$, cf.~\cite[Thm.~4.1]{FL_2002}.

Riesz bases $\{ x_i \}_{i \in I}$ that are standard normalized tight frames at the same time have to be Hilbert bases of the respective Hilbert $A$-module $\mathcal X$, cf.~\cite[Cor.~4.2]{FL_2002}. Note, that the analysis operator $\theta$ of standard Riesz bases of $\mathcal X$ can be non-surjective in the case at least one of the values $\langle x_i,x_i \rangle$ is unequal $1_A$. Even more, two orthogonal Hilbert bases of a certain Hilbert $A$-module $\mathcal X$ over certain C*-algebras can be unitarily non-equivalent because the projection-values can be non-equivalent. For example, the identity of $A={\mathcal X}=B(l_2)$ splits into a countable family of minimal, pairwise orthogonal projections $\{ p_i \}_{i \in \mathbb N}$, and an element $x$ with $\langle x,x \rangle=1_A$ can be split into a weakly converging sum of the elements $\{ p_ix \}_{i \in \mathbb N}$ with $ax= \sum_{i \in \mathbb N} (ap_i) (p_ix)$ for any $a \in A$. For $A={\mathcal X}=M_n({\mathbb C})$ the analogous example works without convergence arguments. 

Any standard Riesz basis of a Hilbert $A$-module $\mathcal X$ is the image of an orthogonal Hilbert basis by an invertible adjointable operator on $\mathcal X$, because any standard frame $\{ x_i \}_{i \in I}$ of $\mathcal X$ is the image of the standard normalized tight frame $\{ S^{-1/2}(x_i) \}_{i \in I}$ by the positive invertible operator $S^{1/2} =(\theta^* \theta)^{1/2}$, cf.~\cite[Thm.~5.3]{FL_2002}. A frame $\{y_i \}_{i \in I}$ is a dual frame of a given (standard) frame $\{ x_i \}_{i \in I}$ in a Hilbert C*-module $\mathcal X$ if the equality $x=\sum_{i \in I} \langle x,y_i \rangle x_i$ holds for every $x \in \mathcal X$. The frame $\{ S^{-1}(x_i) \}_{i \in I}$ is called the {\it canonical dual frame} of a given standard frame $\{ x_i \}_{i \in I}$ of a Hilbert C*-module $\mathcal X$. All other eventually existing dual frames of $\{ x_i \}_{i \in I}$ are called {\it alternative dual frames}. 
Remarkably, standard Riesz bases might admit more than one standard (canonical and alternative) dual frame, cf.~\cite[Ex.~6.4]{FL_2002}, \cite[Ex.~3.4, Thm.~3.11]{HJLM_2008}. 

It does not make sense to define the dual Bessel sequence of a given standard Bessel sequence in an analogous way as for standard frames, because a pair of such two Bessel sequences can be shown to consist of standard frames automatically, cf.~\cite[Prop.~3.8]{HJLM_2008}. Nevertheless, there might exist unbounded in norm dual sequences for some standard frames or Bessel sequences of certain Hilbert C*-modules, cf.~\cite[Ex.~3.2]{HJLM_2008} for an unbounded dual sequence of a standard frame, or consider e.g.~the pair of the Bessel sequence $\{i^{-1} \cdot x_i \}_{i \in I}$ and the (unbounded) sequence $\{ i \cdot y_i \}_{i \in I}$ as a replacement for a given pair of dual to each other standard frames  $\{ x_i \}_{i \in I}$ and $\{ y_i \}_{i \in I}$. 

We want to give a direct proof of the fact that Riesz bases admit a unique Riesz basis among their canonical and alternative dual frames, cf.~\cite[Lemma 3.9, Thm.~3.10, Thm.~3.12, Conj.~3.18]{HJLM_2008}, \cite[Cor.~4.14, Conj.~in Sect.~4]{Jing_2006}. 

\begin{thm}
  Let $A$ be a unital C*-algebra and $\mathcal X$ be a Hilbert $A$-module admitting a standard Riesz basis $\{ x_i \}_{i=1}^\infty$. Let $\{ y_i \}_{i=1}^\infty$  be a dual frame of $\{ x_i \}_{i=1}^\infty$ and a standard Riesz basis of $\mathcal X$ at the same time, i.e. $x=\sum_{i=1}^\infty \langle x,y_i \rangle x_i = \sum_{i=1}^\infty \langle x,x_i \rangle y_i$ for any $x \in \mathcal X$. Then $y_i = S_x (x_i)$, $i \in \mathbb N$, is the canonical dual frame of $\{ x_i \}_{i =1}^\infty$. \newline
  In other words, a standard Riesz basis of $\mathcal X$ admits only one standard Riesz basis among all their standard canonical and alternative dual frames -- the canonical dual frame. 
\end{thm}

\begin{proof}
The Riesz basis and standard frame $\{ x_i \}_{i=1}^\infty$ admits a canonical standard normalized tight frame $\{S_x^{1/2}(x_i) \}_{i=1}^\infty$ which is a standard Riesz basis, too, and an orthogonal Hilbert sequence, cf.~\cite[Prop.~2.5, Cor.~4.2, Thm.~5.3, Cor.~5.7]{FL_2002}. Similarly, the standard Riesz basis $\{ y_i \}_{i=1}^\infty$ has the canonical standard normalized tight frame $\{ S_y^{1/2}(y_i) \}_{i=1}^\infty$ which is an orthogonal Hilbert sequence, too. Apply the respective analysis operators and the rules $P(e_i)=\theta_{x}(\{ S_{x}^{1/2}(x_i) \}_{i=1}^\infty)$, $P(e_i)=\theta_{y}(\{ S_{y}^{1/2}(y_i) \}_{i=1}^\infty)$ for a fixed orthonormal basis $\{ e_i \}_{i=1}^\infty$ of $l_2(A)$ and the common range projection $P: l_2(A) \to \theta({\mathcal X})$, because both the standard frames are supposed to be Riesz bases and, hence, the analysis operators are adjointable, injective and surjective onto $P(l_2(A))$, cf.~\cite[Thm.~4.1, Prop.~6.3, Thm.~7.2]{FL_2002}. Joining these two sets of equalities we get
\begin{equation} \label{eq1} 
   P(e_i) =  \left\langle S_x^{1/2}(x_i),S_x^{1/2}(x_i) \right\rangle S_x^{1/2}(x_i) = \left\langle S_y^{1/2}(y_i),S_y^{1/2}(y_i) \right\rangle S_y^{1/2}(y_i)
\end{equation}
for any $i \in \mathbb N$. If $\{ x_i \}_{i=1}^\infty$ is a standard Riesz basis then the canonical standard normalized frame $\{ S_x^{1/2}(x_i) \}_{i=1}^\infty$ is an orthogonal Hilbert sequence, i.e. the range projection of the element $S_x^{1/2}(x_i)$ is $\langle S_x^{1/2}(x_i),S_x^{1/2}(x_i) \rangle$, i.e. the equality $\langle S_x^{1/2}(x_i),S_x^{1/2}(x_i) \rangle S_x^{1/2}(x_i) = S_x^{1/2}(x_i)$ holds for any $i \in \mathbb N$. The analogous assertion is true for $S_y^{1/2}(y_i)$, $i \in \mathbb N$. So (\ref{eq1}) transforms to
\begin{equation} \label{eq2}
    S_x(x_i) =  S_x^{1/2}S_y^{1/2} (y_i)
\end{equation}
for any $i \in \mathbb N$. Therefore, the two standard (canonical or alternative) dual frames and Riesz bases $\{ S_x(x_i) \}_{i=1}^\infty$ and $\{ y_i \}_{i=1}^\infty$ are similar via the adjointable invertible module operator $S_x^{1/2}S_y^{1/2}$ on $\mathcal X$. By \cite[Prop.~6.7]{FL_2002} the two frames have to coincide, i.e. $S_x^{1/2}S_y^{1/2} = {\rm Id}_{\mathcal X}$ and $y_i=S_x(x_i)$ for any $i \in \mathbb N$.    
\end{proof}


Using the formulae defining the analysis and the synthesis operators of standard frames we want to investigate key properties of analogous operators for standard Bessel sequences. The following two propositions give an alternative characterization of standard Bessel sequences, which are summarized in Theorem \ref{thm-Bessel}. 

\begin{prop}
Let $A$ be a unital C*-algebra and $\mathcal X$ be a Hilbert $A$-module admitting a generating standard Bessel sequence $\{ x_i \}_{i=1}^\infty$. Let $M \in {\mathbb R}$ be the smallest positive number such that $\sum_{i=1}^\infty \langle x,x_i \rangle \langle x_i,x \rangle \leq M \cdot \langle x,x \rangle$ for any $x \in \mathcal X$. Then $\sum_{i=1}^\infty \langle c_ix_i , c_ix_i \rangle \leq M \cdot \sum c_ic_i^*$ for any sequence $\{ c_i \}_{i=1}^\infty \in l_2(A)$.
\end{prop}

\begin{proof}
Consider the linear operator $Q: {\mathcal X} \to l_2(A)$ defined by $Q(x)=\sum_{i=1}^\infty \langle x,x_i \rangle e_i$ for a fixed orthonormal basis $\{ e_i \}_{i=1}^\infty$ of $l_2(A)$. It is $A$-linear. Calculating its norm we have the inequalities
\begin{equation*}
   \langle Q(x),Q(x) \rangle \leq \sum_{i=1}^\infty \langle x,x_i \rangle \langle x_i,x \rangle \leq M \cdot \langle x,x \rangle
\end{equation*}
by the Bessel sequence property of $\{ x_i \}_{i.=1}^\infty$, forcing $\|Q\|^2 \leq M$, so $Q$ is bounded. Furthermore,
\begin{equation*}
   M \cdot \langle x_k,x_k \rangle \geq \sum_{i=1}^\infty \langle x_k,x_i \rangle \langle x_i,x_k \rangle \geq \langle x_k,x_k \rangle\langle x_k,x_k \rangle
\end{equation*}
because we are summing up positive elements. So $\|x_k\| \leq \sqrt{M}$ for any $k \in \mathbb N$. Now, consider partial sums $\{ x_{c,N}=\sum_{k=1}^N c_kx_k \}_{N=1}^\infty$ for sequences $\{ c_k \}_{k=1}^\infty \in l_2(A)$. Consider the transformation
\begin{equation*}
   Q(x_{c,N}) =   \sum_{i = 1}^\infty \left\langle \sum_{k=1}^N c_kx_k, x_i \right\rangle e_i = \sum_{k=1}^N c_k \sum_{i=1}^\infty \langle x_k,x_i \rangle e_i = \sum_{k=1}^N c_k Q(x_k) \, .
\end{equation*}
Estimating $\langle Q(x_{c,N}),Q(x_{c,N}) \rangle$ we get
\begin{eqnarray*}
    \langle Q(x_{c,N}),Q(x_{c,N}) \rangle & = & \sum_{k=1}^N c_k  \left( \sum_{i=1}^\infty \langle x_k,x_i \rangle \langle x_i,x_k \rangle \right) c_k^* \\
                                        & \leq & \sum_{k=1}^N c_k \, (M \cdot \langle x_k,x_k \rangle) \, c_k^* \\
                                        & \leq & M \cdot \sum_{k=1}^N c_k \| x_k \|^2 c_k^* \\
                                        & \leq & M^2 \cdot \sum_{k=1}^N c_kc_k^* \, .
\end{eqnarray*}
Consequently,  $\langle x_{c,N},x_{c,N} \rangle \leq M \cdot \sum_{k=1}^N c_kc_k^*$ for any $N \in \mathbb N$, and $\sum_{i=1}^\infty \langle c_ix_i , c_ix_i \rangle \leq M \cdot \sum c_ic_i^*$ for any sequence $\{ c_i \}_{i=1}^\infty \in l_2(A)$.                         
 \end{proof}
 
 The next assertion is similar to \cite[Lemma 3.1]{KhKh}, however the proof given there is not correct. Also \cite[Lemma 3.10]{Jing_2006}, \cite[Lemma 2.7]{HJM_2009} are not correct, because the $A$-dual Banach $A$-module of the Hilbert $A$-module $l_2(A)$ is $l_2(A)'$, and $l_2(A) \subset l_2(A)'$ for $A$ infinite-dimensional (see last conclusion of that proof), cf.~\cite[Lemma 4.1]{F_1990}, \cite{Gao}. Here we prove a slightly more general fact. As an application of \cite[Lemma 3.10]{Jing_2006} Wu Jing formulated and almost proved analogous facts to the next two statements as  \cite[Prop.~3.11]{Jing_2006}.
 
\begin{prop} \label{prop-Bessel2}
Let $A$ be a unital C*-algebra and $\mathcal X$ be a Hilbert $A$-module admitting a generating standard Bessel sequence $\{ x_i \}_{i=1}^\infty$. Let $\{ x_i \}_{i \in I} \subset \mathcal X$ be a standard generating sequence with the property that there exists a smallest real number $M \in {\mathbb R}$ such that $\left\langle \sum_{i=1}^\infty c_ix_i, \sum_{i=1}^\infty c_ix_i \right\rangle \leq M \cdot  \sum_{i=1}^\infty c_i  c_i^*$ for any sequence $c=\{ c_i \}_{i=1}^\infty \in l_2(A)$. Then $\sum_{i=1}^\infty \langle x,x_i \rangle \langle x_i,x \rangle \leq M \cdot \langle x,x \rangle$ for any $x \in \mathcal X$, i.e. the sequence $\{ x_i \}_{i=1}^\infty$ is a Bessel sequence. 
\end{prop}

\begin{proof}
Let $R: l_2(A) \to \mathcal X$, $R(\{ c_i \}_{i=1}^\infty ) = \sum_{i=1}^\infty c_i x_i$ be an $A$-linear operator. By assumption, it is bounded by $\sqrt{M}$. For any $x \in \mathcal X$ we have the inequality
\begin{eqnarray*}
      \langle \{ \langle x,x_i \rangle\}_{i=1}^\infty, \{ c_i \}_{i =1}^\infty \rangle_{l_2}  \langle \{ \langle x,x_i \rangle\}_{i=1}^\infty, \{ c_i \}_{i =1}^\infty \rangle_{l_2}^* 
           & = & \left( \sum_{i=1}^\infty \langle x,x_i \rangle c_i^*\right)  \left( \sum_{i=1}^\infty \langle x,x_i \rangle c_i^* \right)^* \\
           & = & \langle x, R(\{ c_i \}_{i=1}^\infty) \rangle \langle R(\{ c_i \}_{i=1}^\infty), x \rangle \\
           & \leq & \langle x,x \rangle \cdot \| R(\{ c_i \}_{i=1}^\infty ) \|^2 \\
           & = & \langle x,x \rangle \cdot \left\| \sum_{i=1}^\infty c_ix_i \right\|^2 \\
           & \leq & \langle x,x \rangle \cdot M \cdot \left\| \sum_{i=1}^\infty c_i c_i^* \right\| \, .
\end{eqnarray*}           
Hence, $\{ \langle x,x_i \rangle \}_{i=1}^\infty$ belongs to the $A$-dual Banach $A$-module $l_2(A)'$ of $l_2(A)$. The adjoint operator $R^*$ of $R$ maps the $A$-dual Banach $A$-module ${\mathcal X}'$ of $\mathcal X$ into $l_2(A)'$ and has the same norm as $R$. The Hilbert $A$-module $\mathcal X$ admits a canonical isometrical embedding into its $A$-dual Banach $A$-module ${\mathcal M}'$.  In particular, $R^*(\langle \cdot,x \rangle) = \{ \langle x,x_i \rangle \}_{i=1}^\infty$. Following the construction of W.~L.~Paschke in \cite[Sect.~4, 259-260]{Pa}, form the algebraic tensor product $l_2(A) \otimes_A A^{**}$ of $l_2(A)$ with the bidual W*-algebra $A^{**}$ of $A$ and lift the $A$-valued inner product to an $A$-valued inner product on $l_2(A^{**})$, a Hilbert $A^{**}$-module. Apply \cite[Cor.~4.3, Thm.~3.2]{Pa} to lift the $A^{**}$-valued inner product on $l_2(A^{**})$ further to an $A^{**}$-valued inner product $\langle \cdot,\cdot \rangle_{A^{**}}$ on the $A$-dual, self-dual Hilbert $A^{**}$-module $l_2(A^{**})'$. The norm of the respectively extended operator $R^*$ is preserved, cf.~\cite[Prop.~3.6, Cor.~3.7]{Pa}.  Then
\begin{eqnarray*}
\sum_{i=1}^\infty \langle x,x_i \rangle \langle x_i,x \rangle & = & \langle T^*(x),T^*(x) \rangle_{A^{**}} \\
        & \leq &  \| R^* \|^2 \langle x,x \rangle  \\
        & = & M \cdot \langle x,x \rangle \, .
\end{eqnarray*}        
The sequence $\{ x_i \}_{i=1}^\infty$ is a Bessel sequence. 
\end{proof}

\begin{thm} \label{thm-Bessel}
Let $A$ be a unital C*-algebra and $\mathcal X$ be a Hilbert $A$-module $\mathcal X$ admitting a (standard) Bessel sequence. Let $\{ x_i \}_{i=1}^\infty$ be a generating sequence of $\mathcal X$. Then the following two conditions are equivalent:
\begin{itemize}
   \item[(i)]  The sequence $\{ x_i \}_{i=1}^\infty$ is a standard Bessel sequence with upper bound $M>0$, i.e. 
$\sum_{i=1}^\infty \langle x,x_i \rangle \langle x_i,x \rangle \leq M \cdot \langle x,x \rangle$ (, where the sum converges in norm).
   \item[(ii)] The sequence $\{ x_i \}_{i=1}^\infty$  fulfills the inequality $\left\langle \sum_{i=1}^\infty c_ix_i, \sum_{i=1}^\infty c_ix_i \right\rangle \leq M \cdot  \sum_{i=1}^\infty c_i  c_i^* $ with upper bound $M>0$ (, where sums converge in norm). 
\end{itemize}
\end{thm}


In \cite[Prop.~3.11, Prop.~3.12, Thm.~4.9, Prop.~4.11]{Jing_2006}, \cite[Lemma 3.9, Thm.~3.10]{HJLM_2008} and \cite[Prop.~2.8]{HJM_2009} we find further characterizations of the different types of sequences under discussion in terms of their synthesis operators $\sigma=\theta^*$:

\begin{prop} \label{prop_sequ} {\rm (\cite[Prop.~2.8]{HJM_2009}, \cite[Lemma 3.9, Thm.~3.10]{HJLM_2008})} \newline
  Let $\{ x_i \}_{i=1}^\infty$ be a sequence of a finitely or countably generated Hilbert $A$-module $\mathcal X$ over a unital C*-algebra $A$. Then
  \begin{itemize}
  \item[(i)] $\{ x_i \}_{i=1}^\infty$ is a standard Bessel sequence with Bessel bound $0<M$ if and only if the synthesis operator $\theta^*$ is a well-defined bounded $A$-linear operator with $\| \theta^* \| \leq \sqrt{M}$. 
  \item[(ii)]   $\{ x_i \}_{i=1}^\infty$ is a  standard frame of $\mathcal X$ with frame bounds $0<m \leq M$ if and only if  $\{ x_i \}_{i=1}^\infty$ is a generating set of $\mathcal X$ and the synthesis operator $\theta^*$ is a bounded adjointable operator with the property
\[ 
      \quad \quad \quad \,\, \sqrt{m} \| \{ c_i \}_{i=1}^\infty \| \leq \| \theta^*( \{ c_i \}_{i=1}^\infty ) \| \leq  \sqrt{M} \| \{ c_i \}_{i=1}^\infty \| \,\, {\rm for} \,\, {\rm any} \,\,  \{ c_i \}_{i=1}^\infty \in {\rm Ker}(\theta^*)^\perp  \, .
\]   
In that case $\theta^*$ is surjective, adjointable, and hence, has closed range.    
   \item[(iii)]  $\{ x_i \}_{i=1}^\infty$ is a standard Riesz basis of $\mathcal X$ with frame bounds $0<m \leq M$ if and only if  $\{ x_i \}_{i=1}^\infty$ is a generating set of $\mathcal X$, $\theta^*(\{ c_i \}_{i \in I})=\sum_{i \in I} c_ix_i = 0$ for some set $\{ c_i \}_{i \in  I} \in l_2(A)$ if and only if  ${c_ix_i}_{i \in I}$ is the zero sequence, and the synthesis operator $\theta^*$ is a bounded operator with the property
\[ 
      \quad \quad \quad \,\, \sqrt{m} \| \{ c_i \}_{i=1}^\infty \| \leq \| \theta^*( \{ c_i \}_{i=1}^\infty ) \| \leq  \sqrt{M} \| \{ c_i \}_{i=1}^\infty \| \,\, {\rm for} \,\, {\rm any} \,\,  \{ c_i \}_{i=1}^\infty \in {\rm Ker}(\theta^*)^\perp  \, .
\]
In that case $\theta^*$ is bijective between the orthogonal complement of its kernel in $l_2(A)$  and $\mathcal X$. Its canonical dual frame is a Riesz basis, too.
   \item[(iv)]  $\{ x_i \}_{i=1}^\infty$ is a standard modular Riesz basis of $\mathcal X$ with frame bounds $0<m \leq M$ if and only if  $\{ x_i \}_{i=1}^\infty$ is a generating set of $\mathcal X$ and the synthesis operator $\theta^*$ is a bounded operator with the property
\[ 
       \quad \quad \quad \,\, \sqrt{m} \| \{ c_i \}_{i=1}^\infty \| \leq \| \theta^*( \{ c_i \}_{i=1}^\infty ) \| \leq  \sqrt{M} \| \{ c_i \}_{i=1}^\infty \| \,\, {\rm for} \,\, {\rm any} \,\,  \{ c_i \}_{i=1}^\infty \in l_2(A)  \, .
\]
In that case $\theta^*$ is bijective and  $\{ x_i \}_{i=1}^\infty$ admits a unique dual frame (which is a standard modular Riesz basis, too). 
   \end{itemize}
\end{prop}     

\begin{dfn} ({\rm \cite[Def.~2.1]{Hasannasab}}) \newline
 Let $A$ be a unital C*-algebra and $\mathcal X$ be a finitely or countably generated Hilbert $A$-module. Let $I$ be a finite or countable subset of $\mathbb N$.
 A sequence $\{ x_i \}_{i \in I} \in \mathcal X$ is said to be:
\begin{itemize}
 \item[(i)] {\it $\omega$-independent}, if for convergent sums $\sum_{i \in I} a_ix_i$ with $\{ a_i \}_{i \in I} \in l_2(A)$ their equality to zero forces $a_i = 0$ for any $i \in I$.
 \item[(ii)] {\it relatively $\omega$-independent}, if for convergent sums $\sum_{i \in I} a_ix_i$ with $\{ a_i \}_{i \in I} \in {\rm Ker}(\theta^*)^\perp \subseteq l_2(A)$ their equality to zero forces $a_i = 0$ for any $i \in I$.
 \item[(iii)] {\it biorthogonal to a sequence $\{ y_i \}_{i \in I} \in \mathcal X$}, if one has the equalities $\langle x_i,y_j \rangle = \delta_{ij} 1_A$, with $\delta_{ij} = 0$ for $i \not= j$ and $\delta_{ii}=1$ for any $i$, for any $i,j \in I$.
 \item[(iv)] {\it relatively biorthogonal to a sequence $\{ y_i \}_{i \in I} \in \mathcal X$}, if $\theta_y(y_i) \in {\rm Ker}(\theta_x^*)^\perp$ for any $i \in I$, and one has the equalities $\langle x_i,y_j \rangle = \delta_{ij} p_{x_i}$, with $\delta_{ij} = 0$ for $i \not= j$ and $\delta_{ii}=1$ for $i$, for any $i,j \in I$ and a sequence of non-zero projections $\{ p_{x_i} \}_{i \in I}$ of $A$.
 \item[(v)] {\it exact frame}, if $\{ x_i \}_{i \in I}$ is a frame for $\mathcal X$, and the sequences $\{ x_i \}_{i \in I \setminus \{ k \}}$ cease to be a frame for $\mathcal X$ for every $k \in I$.
\end{itemize}  
\end{dfn}  

We extend \cite[Thm.~2.2]{Hasannasab} from the characterization of modular Riesz bases to a similar characterization of general Riesz bases in Hilbert C*-modules. For partial results in the Hilbert space situation see \cite[Thm.~2.2]{Hasannasab}.

\begin{thm}  
    Let $A$ be a unital C*-algebra and $\mathcal X$ be a finitely or countably generated Hilbert $A$-module. Let $I$ be a finite or countable subset of $\mathbb N$. For a standard frame $\{ x_i \}_{i \in I} \in \mathcal X$ the following assertions are equivalent:
\begin{itemize}
  \item[(i)]  $\{ x_i \}_{i \in I}$ is a Riesz basis (resp., a modular Riesz basis),
  \item[(ii)] $\{ x_i \}_{i \in I}$ is relatively $\omega$-independent (resp., $\omega$-independent),
  \item[(iii)] $\{ x_i \}_{i \in I}$ and its canonical dual frame $\{ S^{-1}(x_i) \}_{i \in I}$ are relatively biorthogonal (resp., biorthogonal),
  \item[(iv)] $\{ x_i \}_{i \in I}$ has a relatively biorthogonal sequence (resp., has a biorthogonal sequence).
\end{itemize}
\end{thm}

\begin{proof} 
The proof follows essentially the lines of the proof of  \cite[Thm.~2.2]{Hasannasab} with a simplification and some changes in the proving techniques. Note, that the analysis operator of the canonical dual $\{ S^{-1}(x_i) \}_{i \in I}$ maps it onto ${\rm Ker}(\theta^*)^\perp \subseteq l_2(A)$. Moreover, the (left and right) carrier projection $p_{x_i}$ of $\langle x_i,x_i \rangle$ in $A \subseteq A^{**}$, calculated in the bidual von Neumann algebra $A^{**}$ of $A$, has the property $p_{x_i} x_i = x_i$ and $p_{x_i} \langle S^{-1}(x_i), x_i \rangle p_{x_i} = \langle S^{-1}(x_i), x_i \rangle$.  This forces $(\langle S^{-1}(x_k),x_k \rangle - 1_A) x_k = (\langle S^{-1}(x_k),x_lk\rangle - 1_{A^{**}}) x_k =  (\langle S^{-1}(x_k),x_k \rangle - p_{x_k}) x_k$. Also, ${\rm Ker}(\theta^*)^\perp=\{ 0 \}$ if and only if $\{ x_i \}_{i \in I}$ is a modular Riesz basis. The units of $A$ and of $A^{**}$ coincide in the universal representation of $A$, i.e. $1_A = 1_{A^{**}}$. The implication (iii) $\to$ (iv) is obvious by definitions. 

(i) $\to$ (ii)  If $\{ x_i \}_{i \in I}$ is a Riesz basis of $\mathcal X$ then the synthesis operator $\theta^*$ maps ${\rm Ker}(\theta^*)^\perp \subseteq l_2(A)$ or ${\rm Ker}(\theta^*)^\perp \subseteq A^n$, $n \in \mathbb N$, bijectively onto $X$, depending on whether $X$ is countably or algebraically finitely generated. Assume that $\sum_{i \in I} a_ix_i = 0$ for some sequence $\{ a_i \}_{i \in I}$ of ${\rm Ker}(\theta^*)^\perp$. By Proposition \ref{prop_sequ} (iv) (resp., (v)) we have the estimate
\[
    \sqrt{m} \| \{ a_i \}_{i=1}^\infty \| \leq \| \theta^*( \{ a_i \}_{i=1}^\infty ) \|  \,\, {\rm for} \,\, {\rm any} \,\,  \{ a_i \}_{i=1}^\infty \in {\rm Ker}(\theta^*)^\perp
\]
for some real constant $0 < m$. Consequently, every single element $a_i$, $i \in \mathbb N$, has to be equal to zero for those sequences, and $\theta^*$ is injective on ${\rm Ker}(\theta^*)^\perp$. For modular Riesz bases ${\rm Ker}(\theta^*) = \{ 0 \}$. 

(ii) $\to$ (iii) For the canonical dual frame $\{ S^{-1}(x_i) \}_{i \in I}$ of the given frame we have $x_k = \sum_{i \in I} \langle S^{-1}(x_k), x_i \rangle x_i$ for any $k \in I$ with norm convergence. We get the equality
\[ 
    \sum_{i \not= k}  \langle S^{-1}(x_k), x_i \rangle x_i +  ( \langle S^{-1}(x_k), x_k \rangle - p_{x_k}) x_k = 0 \, .
\]
The assumed $\omega$-independence property gives
\[
    \langle S^{-1}(x_k), x_k \rangle = p_{x_k} \in A \cap A^{**} \,\,\, {\rm for} \,\, {\rm each} \,\, k \in {\mathbb N} , \,\,  \langle S^{-1}(x_k), x_i \rangle = 0 \,\, {\rm for} \,\, i \not= k \, .
\]
So the biorthogonality of the frames $\{ x_i \}_{i \in I}$ and $\{ S^{-1}(x_i) \}_{i \in I}$ follows. In the case of modular Riesz bases $p_{x_k}=1_A$ for any $k \in I$.

(iv) $\to$ (i)  Let $\{ a_i \}_{i \in I}$ of ${\rm Ker}(\theta^*) \subseteq A^n$, $n \in \mathbb N$, or of ${\rm Ker}(\theta^*) \subseteq l_2(A)$.  Without loss of generality, we may assume $a_k = a_k p_{x_k}$ with $p_{x_k}$ being the carrier projection of $\langle x_k,x_k \rangle$ in $A$ for any $k \in I$, because the complementary sequence $\{ a_k (1_A-p_{x_k}) \}_{i \in I}$ belongs to ${\rm Ker}(\theta^*)$ obviously.  Consider a sequence $\{ y_i \}_{i \in I}$ relatively biorthogonal to $\{ x_i \}_{i \in I}$. Then
\[    
      a_k = a_k p_{x_k} = a_k \langle x_k, y_k \rangle = \langle a_kx_k, y_k \rangle = \left\langle \sum_{i \in I} a_ix_i,y_i \right\rangle = 0
\]
since  $\sum_{i \in I} a_ix_i$ converges in norm and $\{ a_i p_{x_i}\}_{i \in I}$ is assumed to be in ${\rm Ker}(\theta^*)$. Therefore, the synthesis operator $\theta^*$ is bijective from ${\rm Ker}(\theta^*)^\perp$ to $\mathcal X$, and $\{ x_i \}_{i \in I}$ is a Riesz basis of $\mathcal X$. For modular Riesz bases ${\rm Ker}(\theta^*) = \{ 0 \}$ and $p_k = 1_A$ for any $k \in I$. 
\end{proof}


\section{Near-modular Riesz bases}

We want to investigate redundancy of frame elements of several kinds of frames in more detail, i.e.~chances for the erasure of certain frame elements preserving the frame property of the remaining sequence. In other words, we want to look for robustness properties of frames. Knowing the frame theory for Hilbert spaces, the removal of a vector $x_i$ from a frame $\{ x_i \}_{i \in I}$  for a Hilbert space $H$ leaves either a frame or an incomplete set, cf.~\cite[Thm.~5.4.7]{Chr}, \cite{BCHL_2003}, \cite[Thm.~3.14]{Jing_2006}. More precisely for Hilbert spaces $H$, either $\langle x_k, S^{-1}(x_k) \rangle \not= 1$ for a certain index $k$, then $\{ x_i \}_{i \not= k}$ is still a frame for $H$, or $\langle x_k, S^{-1}(x_k) \rangle \ = 1$ for a certain index $k$, then $\{ x_i \}_{i \not= k}$ is not a frame for $H$ anymore. In \cite[Thm.~3.16]{Jing_2006} Wu Jing tried to extend this point of view straight to Hilbert C*-modules contexts, however we have found a counterexample to \cite[Thm.~3.16(ii)]{Jing_2006} (see Example \ref{ex_coeff} below) and are going to present the results for the Hilbert C*-modules context. 

One key point is the possible existence of C*-algebra elements which are not invertible but admit merely a Moore-Penrose inverse. The other key point is that we can identify frame elements which could be eliminated preserving the frame property of the remainder, however we can eliminate them only one by one. We have to assess the situation newly after each single elimination of an element of the initial frame. The reason is the change of the frame operator $S$ after an elimination of a frame element and resulting new characteristics for any single remaining frame element, see Example \ref{ex_coeff} below.

\begin{lem} {\rm (cf.~\cite[Prop.~6.5]{FL_2002}, \cite[Lemma 3.15]{Jing_2006}.)} \newline
Let $A$ be a C*-algebra and $\mathcal X$ be a finitely or countably generated Hilbert $A$-module. Let $\{ x_i \}_{i \in I}$ be a standard frame of $\mathcal X$ with frame operator $S$. Suppose, a certain element $x \in \mathcal X$ admits a decomposition $x=\sum_{i \in I} a_ix_i$ as a norm-convergent sum with coefficients $\{ a_i \}_{i \in I}$ of $A$. Then
\begin{equation} \label{infimum}
    \sum_{i \in I} a_ia_i^* = \sum_{i \in I} \langle x, S^{-1}(x_i) \rangle \langle S^{-1}(x_i),x \rangle + \sum_{i \in I} (a_i - \langle x, S^{-1}(x_i) \rangle)(a_i^* - \langle S^{-1}(x_i),x \rangle) \, .
\end{equation}  
In particular, the infimum of $\sum_{i \in I} a_ia_i^*$ is admitted if $a_i= \langle x, S^{-1}(x_i) \rangle$  for any $i \in I$. Also, the sequences $\{ \langle x, S^{-1}(x_i) \rangle \}_{i \in I}$ and $\{ (a_i - \langle x, S^{-1}(x_i) \rangle ) \}_{i \in I}$ are perpendicular as elements of $l_2(A)$. 
\end{lem}

We need some knowledge about Moore-Penrose inverse elements in C*-algebras $A$. An element $a \in A$ admits a Moore-Penrose inverse element $a^\dagger \in A$ if and only if $aa^\dagger a = a$, $a^\dagger a a^\dagger = a^\dagger$, $(aa^\dagger)^*=aa^\dagger=p$ and $(a^\dagger a)^* = a^\dagger a=q$. The Moore-Penrose inverse is always unique, if it exists. Moreover, $aa^\dagger=p$ and $a^\dagger a=q$ are orthogonal projections in $A$. If $A$ is a unital C*-subalgebra of another C*-algebra $B$ and $a \in A$ admits a Moore-Penrose inverse $a^\dagger$ in $B$, then $a^\dagger \in A$ (inverse closedness), cf.~\cite[Thm.~2(f)]{Roch}, \cite[Thms.~2, 6, 8]{HM_1992}. The unitality of $A$ cannot be omitted. Moreover, $a \in A$ admits a Moore-Penrose inverse in $A$ iff $Aa=A(aa^\dagger a) = A(a^\dagger a)=Aq$ and $aA = (aa^\dagger a)A = (aa^\dagger)A=pA$, cf.~\cite{HM_1992}. For partial results see \cite[Thm.~5.4.7]{Chr}, \cite[Thm.~3.16(i)]{Jing_2006}.

\begin{prop} \label{prop_first} 
Let $A$ be a unital C*-algebra and $\mathcal X$ be a finitely or countably generated Hilbert $A$-module. Let $\{ x_i \}_{i \in I}$ be a standard frame of $\mathcal X$ with frame operator $S$. Denote the (left) carrier projection of an element $x_n$, belonging to the bidual von Neumann algebra $A^{**}$ of $A$, by $p_n$ and recall that $\langle x_n,S^{-1}(x_n) \rangle = p_n \langle x_n,S^{-1}(x_n) \rangle p_n \in p_nA^{**}p_n$. 
\begin{itemize}
\item[(i)] If $(1_A - \langle x_n,S^{-1}(x_n) \rangle )$ is invertible in $A$ then the sequence $\{ x_i \}_{i \not= n, i \in I}$ is still a frame of $\mathcal X$. 
\item[(ii)] If $(1_A - \langle x_n,S^{-1}(x_n) \rangle ) = 0$ then the sequence $\{ x_i \}_{i \not= n, i \in I}$ is incomplete.  In this case the element $S^{-1}(x_n)$ is orthogonal to the elements $\{ x_i \}_{i \not= n, i \in I}$. 
\item[(iii)] If $(p_n - \langle x_n,S^{-1}(x_n) \rangle )$ is invertible in $p_nAp_n \subseteq A^{**}$ (i.e. it admits a Moore-Penrose inverse in $A^{**}$) then its Moore-Penrose inverse and $p_n$ are in $A$, and the sequence $\{ x_i \}_{i \not= n, i \in I}$ is still a frame of $\mathcal X$. 
\item[(iv)] If $p_n \in A$ and $(p_n - \langle x_n,S^{-1}(x_n) \rangle ) = 0$ then the sequence $\{ x_i \}_{i \not= n, i \in I}$ is incomplete.  In this case the element $S^{-1}(x_n)$ is orthogonal to the elements $\{ x_i \}_{i \not= n, i \in I}$. 
\end{itemize}
\end{prop}

\begin{proof}
The proof follows the ideas expressed in the proofs of \cite[Thm.~3.16(i)]{Jing_2006} and \cite[Thm.~3.5.7]{Chr}.
(i): The truncated sequence $\{ x_i \}_{i \not= n, i \in I}$ still has the same or an even lower upper frame bound, so we have only to show the existence of a non-zero lower frame bound for it. By assumption $x_n = \sum_{i \in I} a_ix_i = a_nx_n +  \sum_{i \not= n, i \in I} a_ix_i$, i.e.~$x_n = (1_A - a_n)^{-1} \sum_{i \not= n, i \in I} a_ix_i$ with $a_i := \langle x_n,S^{-1}(x_i) \rangle$ forming a sequence of $l_2(A)$. We apply two times the Cauchy-Schwarz inequality for the Hilbert $A$-modules $A$ and $l_2(A)$ to derive the following estimate for $x \in \mathcal X$:
\begin{eqnarray} \label{f_862}
 & & \\ \nonumber
  &  & \langle x,x_n \rangle \langle x_n,x \rangle
   =  \left\langle x,(1_A-a_n)^{-1} \sum_{i \not= n, i \in I} a_ix_i \right\rangle \cdot \left\langle (1_A-a_n)^{-1} \sum_{i \not= n, i \in I} a_ix_i, x \right\rangle =  \\ \nonumber
   & = & \left\langle x,\sum_{i \not= n, i \in I} a_ix_i \right\rangle ((1_A-a_n)^{-1})^*(1_A-a_n)^{-1} \left\langle \sum_{i \not= n, i \in I} a_ix_i, x \right\rangle \\ \nonumber
   & \leq & \| (1_A-a_n)^{-1} \|^2  \cdot  \left\langle x,\sum_{i \not= n, i \in I} a_ix_i \right\rangle \cdot \left\langle \sum_{i \not= n, i \in I} a_ix_i, x \right\rangle \\ \nonumber
   & = &  \| (1_A-a_n)^{-1} \|^2  \cdot \sum_{i \not= n, i \in I} \langle x, x_i \rangle a_i^*  \cdot  \sum_{i \not= n, i \in I}  a_i \langle x_i, x \rangle \\ \nonumber
   & = &  \| (1_A-a_n)^{-1} \|^2   \cdot \left\langle \{ \langle x,x_i \rangle \}_{i \not= n, i \in I}, \{ a_i \}_{i \not= n, i \in I} \right\rangle_{l_2} \cdot \left\langle \{ a_i \}_{i \not= n, i \in I},  \{ \langle x,x_i \rangle \}_{i \not= n, i \in I} \right\rangle_{l_2} \\ \nonumber
   & \leq & \| (1_A-a_n)^{-1} \|^2  \cdot \left\| \sum_{i \not= n, i \in I} a_ia_i^* \right\| \cdot \sum_{i \not= n, i \in I} \langle x,x_i \rangle \langle x_i,x \rangle \, .
\end{eqnarray}
Therefore, we can estimate the lower frame bound as a result of the inequality
\begin{eqnarray*}
   C \langle x,x \rangle  & \leq & \sum_{i \in I} \langle x,x_i \rangle \langle x_i,x \rangle \\
                                 & = & \langle x,x_n \rangle \langle x_n,x \rangle + \sum_{i \not= n, i \in I} \langle x,x_i \rangle \langle x_i,x \rangle \\
                                 & \leq & \left( \| (1_A-a_n)^{-1} \|^2 \cdot \left\| \sum_{i \not= n, i \in I} a_ia_i^* \right\|  + 1 \right) \cdot \sum_{i \not= n, i \in I} \langle x,x_i \rangle \langle x_i,x \rangle
\end{eqnarray*}                               
which is valid for any $x \in \mathcal X$.

(ii): Furthermore, assume $(1_A - \langle x_n,S^{-1}(x_n) \rangle ) = 0$ for some $n$. Use the denotation  $a_i := \langle x_n,S^{-1}(x_i) \rangle$ again. We have two decompositions of $x_n$, namely $x_n = 1_A x_n$ and $x_n = \sum_{i \in I} a_i x_i$ by the reconstruction formula of the frame $\{ x_i \}_{i \in I}$.  Comparing the two sequences $\{ \delta_{n,i} \}_{i \in I}$ and $\{ a_i \}_{i \in I}$ via Lemma \ref{decomp} we obtain the equality
\begin{eqnarray*}
    1 = \sum_{i \in I} | \delta_{n,i} |^2 & = & \sum_{i \in I} | a_i |^2 + \sum_{i \in I} |a_i - \delta_{n,i} |^2 \\
                       & = & | a_n |^2 + \sum_{i \not= n, i \in I} | a_i |^2 + | a_n-1_A |^2 + \sum_{i \not= n, i \in I} | a_i |^2 \, .
\end{eqnarray*}
Since $a_n=1_A$ is assumed we derive $ \sum_{i \not= n, i \in I} | a_i |^2 = 0$, and hence, we arrive at $a_i= \langle x_n S^{-1}(x_i) \rangle = 0$ for any $i \not= n$. The equality $1_A = a_n = \langle x_n, S^{-1}(x_n) \rangle$ forces $S^{-1}(x_n) \not= 0$. We have found a non-zero element $S^{-1}(x_n) \in \mathcal X$ perpendicular to $\{ x_i \}_{i \not=n, i \in I}$, showing the incompleteness of the truncated sequence. 

(iii): To prove the third assertion repeat the arguments of the first part of the proof replacing $(1_A-a_n)^{-1}$ by the Moore-Penrose inverse $(p_n-a_n)^\dagger$ of $(p_n-a_n)$ calculated with respect to $p_n A p_n \subseteq A^{**}$, where $A$ is considered as its universal $*$-representation in $A^{**}$. As pointed out, since $A$ is assumed to be unital the Moore-Penrose inverse $(p_n-a_n)^\dagger$ of $(p_n-a_n)$  and $p_n$ belong to $A$. The calculation (\ref{f_862}) is still valid with that replacement, and hence, the assertion follows. 

(iv): Replace $\delta_{n,i} \cdot 1_A$ by $\delta_{n,i} \cdot p_n$ and note that $a_n = p_n$. The arguments are analogous to those in part (ii) of the proof.
\end{proof}


Trying to generalize the second assertion of Proposition \ref{prop_first} assuming merely that $(1_A - \langle x_n, S^{-1}(x_n) \rangle)$ is not invertible in $A$, we find an obstacle illustrated by the next example. So, the formulation of \cite[Thm.~3.16(ii)]{Jing_2006} is shown to be false. The more, we can eliminate redundant elements of a frame only one-by-one, not as a larger (e.g. maximal) subset of candidates at once, just following the properties of individual frame elements to be considered for elimination.

\begin{ex} \label{ex_coeff}
Consider the commutative C*-algebra $A={\mathbb C}^2$ realized as the subset ${\rm diag}(a,b)$ with $a,b \in \mathbb C$ in the C*-algebra of all complex-valued $2 \times 2$ matrices. Let ${\mathcal X}=A$ as a Hilbert $A$-module. There it is a maximal abelian subalgebra (MASA) containing the center of the hosting matrix C*-algebra. Set
\[ 
     p_1= \left( \begin{array}{cc}  1 & 0 \\ 0 & 0 \end{array} \right) \, , \,\,\, p_2= \left( \begin{array}{cc}  0 & 0 \\ 0 & 1 \end{array} \right)
\]
and consider the non-modular Riesz basis   $\{ p_1, p_2 \}$ and the frame $\{ p_1, p_2, p_2 \}$ of the Hilbert $A$-module $X$. Calculating the analysis, the synthesis and the frame operators of the three-elements frame we obtain
\begin{eqnarray*}
     \,\, \theta({\rm diag}(a,b)) = (a, b, b) \in A^3 \, , \,\, \theta^*((a,b,c)) = {\rm diag}(a, (b+c))  \, , & & \\
     S ({\rm diag}(a,b)) = \theta^*  \theta  ({\rm diag}(a,b)) = {\rm diag}(a,2b) \, .  \quad \quad \,\, & &
\end{eqnarray*}
Similarly we find these operators for the two-element Riesz basis. Therefore, $S^{-1} ({\rm diag}(a,b)) = {\rm diag}(a, b/2)$ for the three-elements frame, and the coefficients to be considered are $\langle p_i, S^{-1}(p_i) \rangle$, $i=1,2$ for both the frames under consideration.  For the three-elements frame, the first one equals to $p_1$, the second and third one equal to $p_2/2$. Both these coefficients are not invertible in $A$. However, for the two-elements Riesz basis we obtain the same value for $p_1$ for the first element and the value $p_2$ for the second element, the same type of non-invertible values. From the three-elements frame we can remove one of the last two elements preserving the frame property for the remainder, but not the first element and not the two last elements together at once. We have to look for other decision criteria which element(s) to remove and in which order, cf.~Props.~\ref{prop_first} and \ref{prop_second} below.

Calculating the excess values of the frame and of the Riesz basis we get $p_2$ and $0$, respectively, cf.~Def.~\ref{def_excess} below. 
\end{ex}

We are going to give another criterion on frame elements to be candidates for eliminiation preserving the frame property of the remaining sequence in the Hilbert C*-module context. Note, that a numerical (i.e. norm-value based) criterion is insufficient, contrary to the Hilbert space situation. Also, the non-invertibility of the value $\langle x, S^{-1}(x) \rangle$ cannot serve as an indicator alone. We rely on a partial result for Hilbert spaces, cf.~\cite[Cor.~5.7]{BCHL_2003}).

\begin{prop} \label{prop_second}
Let $A$ be a unital C*-algebra and $\mathcal X$ be a finitely or countably generated Hilbert $A$-module. Let $\{ x_i \}_{i \in I}$ be a standard frame of $\mathcal X$ with frame operator $S$ and frame bounds $m,M$. Suppose, $x_i \not= 0$. The standard dual frame is $\{ S^{-1}(x_i) \}_{i \in I}$. Let $y$ be an element of the frame, say $y=x_1$. Then the following statements are equivalent:
\begin{itemize}
\item[(i)] There exists a constant $L>0$ such that the sequence $\{ x_i \}_{i \not= 1, i \in I}$ is a frame for $\mathcal X$ with lower frame bound $L$. 
\item[(ii)] $\langle y, S^{-1}(y) \rangle$ has a norm lower than $1$ and $\langle y, S^{-1}(y) \rangle < \|\langle y, S^{-1}(y) \rangle\| \cdot p_y$, where $p_y$ is the left (and right) carrier projection of $\langle y, S^{-1}(y) \rangle$ in the bidual von Neumann algebra $A^{**}$ of $A$. 
\end{itemize}
Note, that the operator $S$ is calculated with respect to the frame $\{ x_i \}_{i \in I}$. After the elimination of $y=x_1$ the frame operator $S'$ of the remaining frame has to be calculated again what might change the values  $\langle x, S'^{-1}(x) \rangle$  of the remaining frame elements.
\end{prop}

\begin{proof}
It is well-known that for a standard frame $\{ x_i \}_{i \in I}$ of a Hilbert $A$-module $\mathcal X$  with frame bounds $m,M$ and for a bounded invertible adjointable operator $T$ on $\mathcal X$ the resulting sequence $\{ T(x_i) \}_{i \in I}$ is a standard frame with frame bounds $m \cdot \| T^{-1} \|^{-2}$ and $M \cdot \| T \|^2$. Furthermore, for a standard frame with  frame operator $S$ the sequence $\{ S^{-1/2}(x_i) \}_{i \in I}$ is a standard normalized tight frame. So the values $\langle x_i, S^{-1}(x_i) \rangle = \langle S^{-1/2}(x_i),S^{-1/2}(x_i) \rangle$ are positive and fulfill the inequality
\linebreak[4]
$0 < \langle x_i, S^{-1}(x_i) \rangle \leq p_i$, where $p_i$ is the left (and right) carrier projection of that value in the bidual von Neumann algebra $A^{**}$ of $A$. Obviously, $L \leq M$. Beside this, $m \cdot \langle x,x \rangle \leq \sum_{i \in I} \langle x,x_i \rangle \langle x_i,x \rangle = \sum_{i \in I} \langle (S^{-1/2}S^{1/2})(x),x_i \rangle \langle x_i, (S^{-1/2}S^{1/2})(x) \rangle = \langle S^{1/2}(x),S^{1/2}(x) \rangle \leq M \cdot \langle x,x \rangle$ for any $x \in \mathcal X$, so $m \leq \|S^{1/2}\|^2 = \| S \| \leq M$ and $M^{-1} \leq \|S^{-1/2}\|^2 = \| S^{-1} \| \leq m^{-1}$ . 

(i) $\to$ (ii): The frame operator $S$ of the given frame fulfills $\| S^{1/2} \|^2 \leq M$. Thus, the standard frame $\{ S^{-1/2}(x_i) \}_{i \not= 1, i \in I}$ has the lower frame bound $L/M$. More precisely, since the frame $\{ S^{-1/2}(x_i) \}_{i \in I}$ is standard normalized tight for $\mathcal X$ we can estimate
\begin{eqnarray} \label{f864}
    \sum_{i \in I} \langle x,S^{-1/2}(x_i) \rangle \langle S^{-1/2}(x_i),x \rangle - \langle x,S^{-1/2}(x_1) \rangle \langle S^{-1/2}(x_1),x \rangle 
        & & \\ \nonumber
        \quad \quad \geq \langle x,x \rangle - \| S^{-1/2}(x_1) \|^2 \langle x,x \rangle & &  \\ \nonumber
        \quad =   (1-  \| S^{-1/2}(x_1) \|^2) \langle x,x \rangle \quad .& & 
\end{eqnarray}
If we set $x=S^{-1/2}(x_1)$ we obtain the optimal lower frame bound  $(1-\| S^{-1/2}(x_1) \|^2)$  and $L/M \leq  (1 - \| S^{-1/2}(x_1) \|^2)$. Finally, $\langle x_1, S^{-1}(x_1) \rangle \leq (1-L/M) p_1$. 

(ii) $\to$ (i): Assume, that $F = \|\langle x_1, S^{-1}(x_1) \rangle\| < 1$. Then $1 - \| S^{-1/2}(x_1) \|^2 \geq (1 - F) > 0$. Using the estimate (\ref{f864}) we show that $\{ S^{-1/2}(x_i) \}_{i \not= 1, i \in I}$ is a frame for $\mathcal X$ with lower frame bound $(1-F)$. Since $S^{1/2}$ is a bounded invertible adjointable operator with $\|S^{-1/2} \|^2 \leq 1/m$ we conclude that $\{ x_i \}_{i \not= 1, i \in I}$ is a frame for $\mathcal X$ with lower frame bound $L=m(1-F)$. 
\end{proof}

The following observation paves the way towards a generalization of near-Riesz bases of Hilbert spaces to similar structures of countably Hilbert C*-modules. There is no analogue for algebraically finitely generated modules $A^n$, $n \in \mathbb N$.

\begin{lem} \label{decomp}
Let $A$ be a unital C*-algebra. Let $P$ be an orthonormal projection on $l_2(A)$ such that $P(l_2(A))$ is a finitely generated Hilbert $A$-submodule. Then the orthogonal complement $({\rm id}-P)(l_2(A))$ is unitarily isomorphic to $l_2(A)$. 
\end{lem}

\begin{proof}
Since $P(l_2(A))$ is a finitely generated Hilbert $A$-module there exists another finitely generated Hilbert $A$-module $\mathcal N$ and a finite positive integer $n$ such that $P(l_2(A)) \oplus {\mathcal N} = A^n$, cf.~Lemma \ref{SS-DD}. So we have unitary isomorphisms
\begin{eqnarray*}
      ({\rm id}-P)(l_2(A)) \oplus A^n & \cong &  ({\rm id}-P)(l_2(A)) \oplus \left( P(l_2(A)) \oplus {\mathcal N} \right) \quad ({\rm cf. \,\, Lemma \,\, \ref{SS-DD}})\\
                                                    & \cong &  l_2(A) \oplus {\mathcal N}\\
                                                    & \cong &  l_2(A)  \quad ({\rm cf. \,\, Lemma \,\, \ref{K}})\\
                                                    & \cong &  l_2(A) \oplus A^n \, .
\end{eqnarray*}
Therefore, $({\rm id}-P)(l_2(A))  \cong l_2(A)$. Since $l_2(A)$ at the right end contains an infinite orthonormal basis, so  $({\rm id}-P)(l_2(A))$ contains one, too.                                             
\end{proof}

If ${\mathcal M}=P(l_2(A))$ is assumed to be countably, but not algebraically finitely generated as a Hilbert $A$-module then we cannot conclude that $({\rm id}-P)(l_2(A))$ is always unitarily isomorphic to $l_2(A)$. Indeed, by Kasparov's stabilization theorem (Lemma \ref{K}) the orthogonal projection onto the second summand $l_2(A)$ of the sum ${\mathcal M} \oplus l_2(A)=l_2(A)$ would force $\mathcal M$ to be unitarily isomorphic to $l_2(A)$ otherwise, what is far from being always true for any unital C*-algebra $A$ and any countably generated Hilbert $A$-module $\mathcal M$, in general.


We want to define a generalized notion of an excess of a standard frame $\{ x_i \}_{i=1}^\infty$ of finitely or countably generated Hilbert C*-modules. For Hilbert spaces the excess of a frame is the dimension of the kernel of the synthesis operator, ${\rm dim} ({\rm Ker}(\theta^*))$, cf.~\cite[Lemma 4.1]{BCHL_2003}. This is always a non-negative integer or infinity. If the orthogonal projection $P: l_2({\mathbb C}) \to {\rm Ker}(\theta^*)$ is finite-dimensional then the dimension of ${\rm Ker}(\theta^*)$ equals to ${\rm tr}(P)$ since $P$ is a trace class operator and compact. Also, $P=({\rm id}-\theta S^{-1}\theta^*)$ and it is a Hilbert-Schmidt operator, cf.~\cite[Prop.~5.5]{BCHL_2003}. Consequently, for any orthonormal basis $\{ e_i \}_{i=1}^\infty$ of $l_2({\mathbb C})$
\begin{eqnarray} \label{excess1}
   & & \\ \nonumber
    {\rm dim}({\rm Ker}(\theta^*)) = {\rm tr}(P) & = & \sum_{i=1}^\infty \langle e_i,P(e_i) \rangle = \sum_{i=1}^\infty \langle e_i, ({\rm id} -\theta S^{-1}\theta^*)(e_i) \rangle \\ \nonumber
    & =&  \sum_{i=1}^\infty (1 - \langle \theta^*(e_i), S^{-1}  \theta^*(e_i) \rangle) = \sum_{i=1}^\infty (1 - \langle x_i, S^{-1}(x_i) \rangle) \, .
\end{eqnarray}
Note, that the frame $\{ S^{-1}(x_i) \}_{i=1}^\infty$ is the canonical dual frame of the given frame $\{ x_i \}_{i=1}^\infty$. 

The dimension of a general algebraically finitely generated Hilbert C*-module as well as the trace of an orthogonal projection operator onto it are not defined, generally speaking. However, the three most right terms of the equation (\ref{excess1}) are well-defined and still equal for e.g.~standard frames of finitely or countably generated Hilbert C*-modules over unital C*-algebras with ${\rm Ker}(\theta^*)$ algebraically finitely generated. The more, the most right term is an expression in the context of the original Hilbert C*-module $\mathcal X$ and of the given standard frame for $\mathcal X$. It does not depend on the choice of the orthonormal basis $\{ e_i \}_{i=1}^\infty$ of $l_2(A)$. 

\begin{dfn} {\rm (cf.~\cite{CC,ACP,BB,Heil,HY_2023}.)}  \label{def_excess} \newline
    Let $A$ be a unital C*-algebra and $\mathcal X$ be a finitely or countably generated Hilbert $A$-module. Let $\{ x_i \}_{i \in I}$ be a standard frame of $\mathcal X$ with frame operator $S$. Let $p_i$ denote the left (and right) carrier projection of  $\langle x_i, S^{-1}(x_i) \rangle$ in $A^{**}$. Then the excess of that frame is defined as the value
\[
        e(\{ x_i \}_{i=1}^\infty) =  \sum_{i=1}^\infty (p_i - \langle x_i, S^{-1}(x_i) \rangle)
\]     
whenever this value is a finite element of the bidual von Neumann algebra $A^{**}$ of $A$ via the universal $*$-representation of $A$ in $A^{**}$. It is a sum of non-negative elements. In that case the supremum of the partially ordered set of all finite partial sums at the right side either exists and is a positive element of $A^{**}$, or the excess equals to infinity by definition. 
\end{dfn}   

\begin{dfn} \label{def_near_Riesz_basis}
Let $A$ be a unital C*-algebra and $\mathcal X$ be a finitely or countably generated Hilbert $A$-module admitting a Riesz basis. A standard frame $\{ x_i \}_{i \in I}$ of a countably generated Hilbert $A$-module $\mathcal X$ is a {\it near-Riesz basis of $\mathcal X$} if the kernel of the synthesis operator is an algebraically finitely generated Hilbert $A$-submodule of $l_2(A)$. (Here the cardinality of the orthonormal basis of the target Hilbert C*-module of the analysis operator equals the cardinality of the frame under consideration, i.e.~one has to take eather $A^k$ for a finite $k \in {\mathbb N}$ or $l_2(A)$.)
\end{dfn}

Note, that the number of generators of ${\rm Ker}(\theta^*)^\perp$ is not unique, in general. What is more, for some C*-algebras $A$ the number of generators could be infinite for certain selected generator sets, depending on the inner structure of $A$. 

The structure of the analysis and of the synthesis operator of near-Riesz bases reminds characteristics of bounded C*-linear Fredholm operators. This was pointed out for the Hilbert space case in \cite{BM_2025} by D.~Biswas and M.~Mitkovski recently. Adjointable C*-linear Fredholm operators between countably generated Hilbert C*-modules over unital C*-algebras were investigated in \cite{MF} by A.~S.~Mishchenko and A.~T.~Fomenko in 1979. 

\begin{dfn}  (cf.~\cite[Def., p.~96]{MF}.) \label{dfn_Fredholm} \newline 
An adjointable operator $T: {\mathcal M} \to {\mathcal N}$ between two countably generated Hilbert $A$-modules ${\mathcal M}$, ${\mathcal N}$ over a unital C*-algebra $A$  is said to be {\it $A$-Fredholm} if there exist topological direct sum decompositions ${\mathcal M} = {\mathcal M}_1 \dotplus {\mathcal M}_2$, ${\mathcal N} = {\mathcal N}_1 \dotplus {\mathcal N}_2$ such that 
\begin{equation}
     T = \left( \begin{array}{cc} T_1 & 0 \\ 0 & T_2 \end{array} \right) : {\mathcal M}_1 \oplus {\mathcal M}_2 \to {\mathcal N}_1 \dotplus {\mathcal N}_2 \, 
\end{equation}
$T_1 : {\mathcal M}_1  \to {\mathcal N}_1$ is a unitary $A$-linear isomorphism,  $T_2: {\mathcal M}_2 \to {\mathcal N}_2$ is a map between two algebraically finitely generated Hilbert $A$-modules and ${\mathcal M}_1$ is orthogonal to ${\mathcal M}_2$.
\end{dfn}

In the situation of near-Riesz bases their analysis operator is injective $A$-Fredholm and their synthesis operator is surjective $A$-Fredholm with orthogonal components ${\mathcal N}_1$ and ${\mathcal N}_2$. $A$-Fredholm operators possess an invariant -- their index. It can be calculated $K$-theoretically as ${\rm index}(T)=[{\mathcal M}_2]-[{\mathcal N}_2]$ in the $K_0(A)$-group of $A$, where $[{\mathcal K}]$ denotes the $K$-theoretic class of the finitely generated Hilbert $A$-module $\mathcal K$ in $K_0(A)$, cf.~\cite[{\S}1, Lemma 2.2]{MF}, \cite[Thm.~15]{MT_2023}. Note, that $K_0(A)$ could be trivial even its generating monoid of all countably generated Hilbert $A$-modules over $A$ might not be trivial, e.g.~for some types of von Neumann factors $A$. This depends on the inner structure of $A$. 

\begin{prop}
Let $A$ be a unital C*-algebra and $\mathcal X$ be a finitely or countably generated Hilbert $A$-module admitting a Riesz basis. Consider a near-Riesz basis $\{ x_i \}_{i \in I}$ of $\mathcal X$.
\begin{itemize}
\item[(i)] The kernel of the synthesis operator ${\rm Ker}(\theta^*)$ of the near-Riesz basis $\{ x_i \}_{i \in I}$ does not depend on the choice of the orthonormal basis $\{ e_i \}_{i \in I}$ of $l_2(A)$ up to unitary isomorphism. The excess of $\{ x_i \}_{i \in I}$ is independent of the decomposition of $S^{-1}$ as $(\theta^* \theta)^{-1}$ and of the choice of the respective orthogonal basis of $l_2(A)$ used for the construction. It is an invariant of the unitary equivalence class of the algebraically finitely generated Hilbert $A$-module ${\rm Ker}(\theta^*)$. 
\item[(ii)] The excess $e(\{ x_i \}_{i \in I})$ of $\{ x_i \}_{i \in I}$ equals to $\sum_{i \in I} \langle e_i,P_\theta(e_i) \rangle$ for any orthonormal basis $\{ e_i \}_{i \in I}$ of $l_2(A)$ used for the construction of $\theta$ and for the respective orthogonal projection $P_\theta: l_2(A) \to {\rm Ker}(\theta^*)$. 
\end{itemize}
\end{prop}

\begin{proof}
We give the proof for countably generated Hilbert C*-modules and infinite frames. For algebraically finitely generated Hilbert C*-modules the target Hilbert C*-module of the analysis operator in most cases has to be replaced by a Hilbert C*-module $A^k$ for a finite $k \in {\mathbb N}$. The operator-theoretic arguments are unchanged.
(i): Let $\{ e_i \}_{i \in I}$ and $\{ f_i \}_{i \in I}$ be two orthonormal bases of $l_2(A)$ and $\theta_e$ and $\theta_f$ the respective analysis operators of $\{ x_i \}_{i \in I}$. Let $U: l_2(A) \to l_2(A)$ be a unitary $A$-linear operator with $U(f_i)=e_i$ for any $i \in I$. Then $U^*\theta_e = \theta_f$, and $U^*$ maps the orthogonal complement of ${\rm Im}(\theta_e)$ onto the orthogonal complement of ${\rm Im}(\theta_f)$, as $U$ does vice versa. That means $\ker \theta_e^* = U^*(\ker \theta_f^*)$.  Moreover, $U^*\theta_e = \theta_f$ implies $\theta_e^* U = \theta_f^*$, $S_f = \theta_f^* \theta_f = \theta_e^* U U^* \theta_e = \theta_e^* \theta_e = S_e$, so the frame operator $S$ is independent of the choice of the orthonormal basis of $l_2(A)$ used for the construction of $\theta$. The excess value in $A^{**}$ is, therefore, an invariant of the unitary equivalence class of ${\rm Ker}(\theta^*)$.

(ii): By the Hilbert C*-module analog of the equality (\ref{excess1}) we have the equality $\sum_{i \in I} \langle e_i,P_\theta(e_i) \rangle = e(\{ x_i \}_{i \in I})$ and the right side sum does not depend on the choice of the orthonormal basis of $l_2(A)$.
\end{proof}

\begin{cor} \label{cor_A-Fredholm}
Let $A$ be a unital C*-algebra. 
A sequence $\{ x_i \}_{i \in I}$ of a countably generated Hilbert $A$-module $l_2(A)$ is a {\it modular near-Riesz basis of $l_2(A)$} if and only if the synthesis operator $\theta^*$ is a surjective $A$-Fredholm operator, equivalently, iff the analysis operator is injective $A$-Fredholm. The excess of a near-Riesz basis can be interpretated as an invariant of the respective analysis and synthesis $A$-Fredholm operators, independent of the choice of an orthonormal basis of $l_2(A)$. Vice versa, a near-Riesz basis of a countably generated Hilbert $A$-module over a unital C*-algebra $A$ admits the invariant $[{\rm Ker}(\theta^*)] \in K_0(A)$ for any possible choice of the analysis operator $\theta$ or, respectively, synthesis operator $\sigma=\theta^*$. 
\end{cor}

\begin{conj} \label{conj1}
Let $A$ be a unital C*-algebra. For a near-Riesz basis of an algebraically finitely or countably generated Hilbert $A$-module, the value $e(\{ x_i \}_{i \in I})$ is finite and belongs to the positive part of $A$.
\end{conj}


\section{Pseudo-frames, pseudo-Riesz sequences and pseudo-Riesz bases}

The goal of this section is to generalize the notions of pseudo-frames, pseudo-Riesz sequences and pseudo-Riesz bases from the context of Hilbert spaces to the context of Hilbert C*-modules following ideas by D.~Biswas and M.~Mitkowski \cite{BM_2025}. We obtain a correspondence of the respective module frame properties to the properties of classes of (semi-)C*-Fredholm operators, what connects these two research areas in a favourable way. 

The Fredholm and semi-Fredholm theory on Hilbert and Banach spaces started by studying the integral equations introduced by E.~I.~Fredholm in 1903 in \cite{Fred_1903}. After that the abstract theory of Fredholm and semi-Fredholm operators on Banach spaces was further developed in numerous papers and books such as \cite{A_2004,A_2018} and \cite{ZZ_2023}. C*-Fredholm theory on Hilbert C*- modules as a generalization of Fredholm theory on Hilbert spaces was started by A.~S.~Mishchenko and A.~T.~Fomenko in \cite{MF} in 1979. They introduced the notion of a C*-Fredholm operator on the standard
module $l_2(A)$ over a unital C*-algebra $A$ and proved the respective generalization of Atkinson's theorem. One of the main differences between classical Fredholm operators on Hilbert spaces and common C*-Fredholm operators on Hilbert C*-modules is that C*-Fredholm operators occasionally might have non-closed images, whereas classical Fredholm operators always have closed images. Semi-C*-Fredholm operators are canonical generalizations of C*-Fredholm operators, as semi-Fredholm operators on Hilbert spaces are generalizations of Fredholm operators there. Inspired by the approach taken in \cite{MF} we are going to define (semi-)C*-Fredholm operators in terms of their matricial decompositions. It turns out that these operators are exactly those that are one-sided invertible modulo compact operators if we consider the standard module $l_2(A)$, but not always in the case of arbitrary Hilbert C*-modules,  cf.~\cite{Ivkovic_2021}. So, for $l_2(A)$ the matricial definition of C*-Fredholm operators coincides with the definition given in \cite{MF}. The interest for studying such operators comes from pseudodifferential operators acting on manifolds with coefficients in some C*-algebra. The classical theory works well for compact manifolds, but not for more general ones. Sometimes, kernels and cokernels of many such (semi-)C*-Fredholm operators are infinite-dimensional as Banach spaces, but become finitely generated viewed as Hilbert C*-modules. This is the most important reason for studying semi-C*-Fredholm operators. We continue studies made in \cite{Ivkovic_2019,Ivkovic_2021,Ivkovic_2024_1} by the second author. Complementary results can be found in \cite{AH_2009,HA_2012}.

\begin{dfn} {\rm (cf.~\cite[Def.~3.1]{BM_2025})}   \label{def_pseudo_frame} \newline
Let $A$ be a unital C*-algebra. A Bessel sequence for a Hilbert $A$-module $\mathcal X$ is a {\it pseudo-frame} if it becomes a frame after adding finitely many appropriate vectors spanning an algebraically generated Hilbert $A$-submodule of $\mathcal X$. 
\end{dfn}

\begin{dfn} {\rm (cf.~\cite[Def.~3.1.1]{Ivkovic_2021})} \label{def_upper_semi} \newline
Let $A$ be a unital C*-algebra and $\mathcal X$ be a countably generated Hilbert $A$-module. An adjointable bounded $A$-linear operator $T$ on $\mathcal X$ is said to be {\it upper semi-$A$-Fredholm} if there exist topological direct decompositions ${\mathcal X}= {\mathcal M}_1  \dotplus {\mathcal N}_1$ and ${\mathcal X}= {\mathcal M}_2  \dotplus {\mathcal N}_2$ for Hilbert $A$-submodules ${\mathcal M}_1, {\mathcal M}_2, {\mathcal N}_1, {\mathcal N}_2$ such that one has a decomposition
\[
       T = \left( \begin{array}{cc} T_1 & 0 \\ 0 & T_2 \end{array} \right)
\]
with $T_1: {\mathcal M}_1 \to {\mathcal M}_2$ a isomorphism, ${\mathcal N}_1$ algebraically finitely generated.       
An adjointable bounded $A$-linear operator $T$ on $\mathcal X$ is said to be {\it lower semi-$A$-Fredholm} if there exist topological direct decompositions ${\mathcal X}= {\mathcal M}_1  \dotplus {\mathcal N}_1$ and ${\mathcal X}= {\mathcal M}_2  \dotplus {\mathcal N}_2$ for Hilbert $A$-submodules ${\mathcal M}_1, {\mathcal M}_2, {\mathcal N}_1, {\mathcal N}_2$ such that one has a decomposition
\[
       T = \left( \begin{array}{cc} T_1 & 0 \\ 0 & T_2 \end{array} \right)
\]
with $T_1: {\mathcal M}_1 \to {\mathcal M}_2$ a isomorphism, ${\mathcal N}_2$ algebraically finitely generated.        
\end{dfn}

By \cite[Lemma 3.1.3]{Ivkovic_2021} $T$ reduced to $\{ 0 \} \dotplus {\mathcal N}_1$ results in the entire module ${\mathcal N}_2$, i.e. $T_2$ is surjective in both cases. 
In the case of upper semi-$A$-Fredholm operators, since algebraically finitely generated Hilbert $A$-modules over unital C*-algebras $A$ are self-dual, $T_2$ is adjointable. Therefore, ${\mathcal N}_2$ is an orthogonal direct summand of $\mathcal M$ by \cite[Thm.~15.4.2, Rem.~15.4.3, Cor.~15.4.8]{NEWO}.
For lower semi-$A$-Fredholm operators ${\mathcal N}_2$ is a direct orthogonal summand of $\mathcal M$ because it is supposed to be self-dual, in particular. 

In the next theorem, we characterize pseudo-frames in $l_2(A)$ by certain C*-Fredholm properties of their synthesis operators.
	
\begin{thm}  \label{prop_pseudo1}
Let $A$ be a unital C*-algebra. A Bessel sequence $\{x_n\}_{n \in {\mathbb N}} \subset l_2(A)$ is a pseudo-frame for $l_2(A)$ if and only if the associated synthesis operator is lower semi-$A$-Fredholm.
\end{thm}
	
\begin{proof} 
The proof is motivated by the proof of  \cite[Proposition 3.1]{BM_2025}. Suppose first that $\{ x_i \}_{i \in {\mathbb N}}$ be a pseudo-frame for $l_2(A)$.  Then there exists a finite set of elements $\{ y_j \}_{j=1}^m$ of $l_2(A)$ such that the sequence $\{ y_1, \ldots y_m, x_1, \ldots , x_i, \ldots \}$ is a frame of $l_2(A)$. By the same arguments as in the proof of \cite[Proposition 3.1]{BM_2025}, we deduce that the synthesis operator $\theta_1^*$ of that frame can be written as $\theta_1^* = \theta^*B^m + GQ_m$, where $\theta^*$ denotes the synthesis operator of the pseudo-frame  $\{ x_i \}_{i \in {\mathbb N}}$, $B$ the backward shift on $l_2(A)$, $Q_m$ the orthogonal projection onto the $A$-submodule $L_m$ generated by the first $m$ orthonormal basis elements of $l_2(A)$, and $G: L_m \to  l_2(A)$ is defined as $G((a_1 \ldots , a_m)) = \sum_{j=1}^m a_j y_j$. By Proposition \ref{prop_sequ}, $\theta_{1}^{*}$ is surjective.	Consequently, it is lower semi-$A$-Fredholm by \cite[Lemma 12]{Ivkovic_2020_2}. Since the set of lower semi-$A$-Fredholm operators is invariant under compact perturbations, cf.~\cite[Theorem 2.3]{Ivkovic_2019}, and since lower semi-$A$-Fredholm operators correspond to right-invertible elements in ${\rm End}_A^*(l_2(A)) / {\rm K}_A(l_2(A))$ we have that $ \theta^{*}B^{m}$ is lower semi-$A$-Fredholm. Hence, by \cite[Corollary 2.6]{Ivkovic_2019} we conclude that $\theta^{*}$ is lower semi-$A$-Fredholm.
		
Conversely, suppose that the synthesis operator $\theta^{*}$ of the Bessel sequence $\{x_n\}_{n \in {\mathbb N}}$ is lower semi-$A$-Fredholm. So, it has a matrix representation
\[
		\begin{pmatrix}
			\theta^{*}_{1} & 0 \\
			0 & \theta^{*}_{4}
		\end{pmatrix}
\]
with respect to a decomposition
\[
		l_2(A) = {\mathcal M}_{1} \dotplus {\mathcal N}_{1}
		\overset{\theta^{*}}{\longrightarrow}
		{\mathcal M}_{2} \dotplus {\mathcal N}_{2} = l_2(A),
\]
where $\theta^{*}_{1}$ is an isomorphism and ${\mathcal N}_{2}$ is finitely generated. Let $\{y_{1},\ldots,y_{m}\}$ be an algebraic generator set of ${\mathcal N}_{2}$ and $G: L_{m} \to {\mathcal N}_{2}$ be given by $G((a_{1},\ldots,a_{m})) = \sum_{j=1}^{m} a_{j}\cdot y_{j}$.
The synthesis operator of $\omega=\{y_{1},\ldots,y_{m},f_{1},f_{2},\ldots\}$ is $\theta_\omega^*=\theta^{*}B_{m}+GQ_m$ again. Let $v \in {\mathcal N}_{2}$ be given. Since $G$ is surjective (onto ${\mathcal N}_{2}$), there exists some $x_{0} \in L_{m}$ such that $G(x_{0})=v$. Then
\[
		\theta_\omega^*(x_{0}) =  \theta^{*}B^{m}(x_{0})+GQ_m (x_{0}) = 0 + G(x_{0}) = v \, ,
\]
so ${\mathcal N}_{2}\subseteq {\rm Im}(\theta_\omega^*)$, since $v \in {\mathcal N}_{2}$ was chosen arbitrarily. Let now $w \in {\mathcal M}_{2}$ be given. Since $\theta^{*}_{1}$ is an isomorphism from ${\mathcal M}_{1}$ onto ${\mathcal M}_{2}$, there exists some $x_{1} \in {\mathcal M}_{1}$ such that $\theta^{*}(x_{1}) = \theta^{*}_{1}(x_{1}) = w$. Hence, we get that
\[
		\theta_\omega^*(F^{m}(x_{1})) = \theta^{*}B^{m}F^{m}(x_{1}) + GQ_m F^{m}(x_{1}) = \theta^{*}B^{m}F^{m}(x_{1}) = \theta^{*}(x_{1}) = w \, ,
\]
so ${\mathcal M}_{2}\subseteq {\rm Im}(\theta_\omega^*)$, since $w \in {\mathcal M}_{2}$ was chosen arbitrarily. Thus, ${\mathcal M}_{2} + {\mathcal N}_{2} \subseteq {\rm Im}(\theta_\omega^*)$, since ${\mathcal M}_{2}\subseteq {\rm Im}(\theta_\omega^*)$, ${\mathcal N}_{2}\subseteq {\rm Im}(\theta_\omega^*)$, and ${\rm Im}(S)$ is a submodule of $l_2(A)$. However,
\[
		{\mathcal M}_{2}+{\mathcal N}_{2} = {\mathcal M}_{2}\dotplus {\mathcal N}_{2} = l_2(A) \, ,
\]
hence ${\rm Im}(\theta_\omega^*) = l_2(A)$, so $\theta_\omega^*$ is surjective. Thus, $\omega=\{y_{1},\ldots,y_{m},f_{1},f_{2},\ldots\}$  is a frame.	
\end{proof}

		
Motivated by the concept of pseudo-Riesz sequences on Hilbert spaces introduced in \cite[Definition 4.1]{BM_2025}, we provide a generalization of this concept in the setting of Bessel sequences in $l_2(A)$.

\begin{dfn}  {\rm (cf.~\cite[Def.~4.1]{BM_2025})}   \label{def_pseudo} \newline
Let $A$ be a unital C*-algebra. A {\it pseudo-Riesz sequence} is a Bessel sequence $\{x_n\}_{n \in {\mathbb N}} \subseteq l_2(A)$ for which there exists a number $m \in {\mathbb N}$ such that $\{x_n\}_{m+1}^{\infty}$ is a modular Riesz basis for $\overline{{\rm Span}_{A} \left\{ x_n \;\middle|\; m+1 \leq n,\; n \in {\mathbb N} \right\}}$. 
\end{dfn}

That is, for pseudo-Riesz sequences $\{x_n\}_{n \in {\mathbb N}} \subseteq l_2(A)$ there exists a positive number $C$ such that $C \cdot \sum_{i=m+1}^\infty |c_i|^2 \leq \| \sum_{i=m+1}^\infty c_i x_i \|^2$ for any $\{ c_i \}_{i \in I} \in l_2({\mathbb C})$, in the case of $m$ removed elements and rearrangement of the Bessel sequence aligning the elements to be removed to the beginning of the sequence. 

\begin{ex}
In contrast to the Hilbert space situation a single sequence element can already make a difference, for example, forcing the synthesis operator to be non-surjective and without norm-closed image. Consider the C*-algebra $A=B(l_2)$ of all bounded linear operators on the Hilbert space $l_2$. Fix an orthonormal basis. Define an operator $a_0 \in A$ by $a_0(x) = \{ 1/i^2 \cdot x_i \}_{i \in \mathbb N}$ for $x = \{ x_i \}_{i \in {\mathbb N}} \in l_2$ with respect to the fixed basis. Then the set $Aa_0$ consists of compact elements of $A$, however e.g.~the element $\{ 1/i \cdot x_i \}_{i \in \mathbb N}$ is not contained in $Aa_0$, so the set is not norm-closed.

Consider the Hilbert $A$-module $l_2(A)$ with its canonical orthonormal basis $\{ e_i \}_{i \in \mathbb N}$. Take the pseudo-Riesz sequence $(a_0, e_2, e_3, \ldots ) \subset l_2(A)$. The synthesis operator $\theta^*$ of it is not surjective and its image is not norm-complete. 
\end{ex}

In the next theorem, we characterize pseudo-Riesz sequences in $l_2(A)$ in terms of their synthesis operators.

\begin{thm}  \label{prop_pseudo2}
Let $A$ be a unital C*-algebra. A Bessel sequence $\{x_n\}_{n \in {\mathbb N}} \subset l_2(A)$ is a pseudo-Riesz sequence for $l_2(A)$ if and only if the associated synthesis operator is upper semi-$A$-Fredholm.
\end{thm}
	
\begin{proof}
Let $L_m$ be the $A$-submodule generated by the first $m$ orthonormal basis elements of $l_2(A)$. Suppose first that the synthesis operator $\theta^*$ of the Bessel sequence $\{ x_i \}_{i \in {\mathbb N}}$ is upper semi-$\mathcal{A}$-Fredholm. Let 
\[
	  	l_2(A) = {\mathcal M}_{1} \dotplus {\mathcal N}_{1} \overset{\theta^{*}}{\longrightarrow} {\mathcal M}_{2} \dotplus {\mathcal N}_{2} = l_2(A)
\]
be decomposition for $\theta^{*}$ as in Definition \ref{def_upper_semi}. Then, ${\mathcal N}_{1} $ is finitely generated, and hence, by the proof of \cite[Theorem 2.7.9]{MT}, \cite[Lemma 2.7.11]{MT}, \cite[Lemma 2.7.13]{MT} and \cite[Lemma 2.16]{Ivkovic_2019}, there exist some $m \in {\mathbb N}$ and some algebraically finitely generated (closed) submodule ${\mathcal P} \subseteq l_2(A)$ such that	
\[
	  	l_2(A) = L_{m}^{\perp} \dotplus ({\mathcal P} \dotplus {\mathcal N}_{1})
	  	\overset{\theta^{*}}{\longrightarrow}
	  	\theta^{*}(L_{m}^{\perp}) \dotplus \big(\theta^{*}({\mathcal P}) \dotplus {\mathcal N}_{2}\big)
	  	= l_2(A) \, ,
\]	  	
is a decomposition for $\theta^{*}$ as in Definition \ref{def_upper_semi}, and $\theta^{*}$ realizes an Isomorphism of $\mathcal P$ and of $\theta^{*}({\mathcal P})$. Consequently, $\theta^{*}$ reduced to $L_m^{\perp}$ is an adjointable $A$-linear isomorphism of $L_m^{\perp}$ onto its image $\theta^*(L_m^{\perp})$. Let $F$ denote the forward shift on $l_2(A)$, then $\theta^*F^m$ is an adjointable $A$-linear isomorphism of $l_2(A)$ onto $\theta^*(L_m^{\perp})$, and for the elements of the associated standard orthonormal basis of  $l_2(A)$ we have $\theta^*F^m(e_i) = x_{m+i}$, $i \in \mathbb N$. Thus, $\{ x_{m+i} \}_{i=1}^\infty$ is a modular Riesz basis for $\overline{   {\rm Span}_{A} \left\{ x_n \;\middle|\; m+1 \leq n,\; n \in {\mathbb N} \right\}  }$.

Conversely, suppose that $\{ x_i \}_{i \in {\mathbb N}}$ is a pseudo-Riesz sequence for $l_2(A)$ and let  $m \in {\mathbb N}$ be such that $\{x_n\}_{m+1}^{\infty}$ is a modular Riesz basis for 
\[
            \overline{{\rm Span}_{A} \left\{ x_n \;\middle|\; m+1 \leq n,\; n \in {\mathbb N} \right\}} \, .
\]
The synthesis operator $\theta_1^*$ of this modular Riesz basis can be expressed as $\theta_1^* = \theta^* F^m$, where $\theta^*$ is the synthesis operator of the pseudo-Riesz sequence  $\{ x_i \}_{i \in {\mathbb N}}$ and $F$ is the forward shift operator on $l_2(A)$. Hence
\begin{equation} \label{(2.1)}
	  		\theta_{1}^{*}B^{m} = \theta^{*}F^{m}B^{m} = \theta^{*}Q_{m} \, , 
\end{equation}
where $Q_{m}$ denotes the orthogonal projection onto $L_{m}^{\perp}$. Since $\theta_{1}^{*}B^{m}$ is bounded below on $L_{m}^{\perp}$, because $B^{m}$ is an isomorphism from $L_{m}^{\perp}$ onto $l_2(A)$, and since $\theta_{1}^{*}$ is an isomorphism from $l_2(A)$ onto $
\overline{{\rm Span}_{A} \left\{ 	x_{m+1},x_{m+2}, \ldots \right\}},$, it follows from \eqref{(2.1)} that $\theta^{*}$ reduced to $L_{m}^{\perp}$ is bounded below. By \cite[Lemma 3.1]{Ivkovic_2019} $\theta^{*}$ is upper semi-$A$-Fredholm.	
 \end{proof}

Motivated by the concept of pseudo-Riesz bases in Hilbert spaces introduced in \cite[Definition 5.1]{BM_2025}, we now give a generalization of this concept in the setting of Bessel sequences in $l_2(A)$.
	  
\begin{dfn}  {\rm (cf.~\cite[Def.~5.1]{BM_2025})} \label{def_pseudo_Riesz_basis}
Let $A$ be a unital C*-algebra. A {\it pseudo-Riesz basis} is a Bessel sequence $\{x_n\}_{n \in {\mathbb N}}$ $\subset l_2(A)$ for which there exists a number $m \in {\mathbb N}$ such that $\{x_n\}_{m+1}^{\infty}$ is a modular Riesz basis for $\overline{{\rm Span}_{A}  \left\{ x_n \;\middle|\; m+1 \leq n,\; n \in {\mathbb N} \right\}}$ and $\left( \overline{{\rm Span}_{A} \left\{  x_n \right\}_{n=m+1}^{\infty}} \right)^{\perp}$ is finitely generated.
\end{dfn}

 In the next theorem, we characterize pseudo-Riesz bases in $l_2(A)$ by the C*-Fredholm property of their synthesis operators. For a related Hilbert space theory result see \cite[Prop.~3.5]{ACP}.

\begin{thm} \label{prop_pseudo3}
Let $A$ be a unital C*-algebra. A Bessel sequence $\{x_n\}_{n=1}^{\infty}\subset l_2(A)$ is a pseudo-Riesz-basis if and only if the associated synthesis operator is $A$-Fredholm.
\end{thm}

\begin{proof}
Let $L_m$ be the $A$-submodule generated by the first $m$ orthonormal basis elements of $l_2(A)$. In the same way as in the proof of Theorem \ref{prop_pseudo2} we can show that $\{x_n\}_{n=m+1}^{\infty}$ is a modular Riesz-basis for $\overline{{\rm Span}_{A} \left\{ x_n  \right\}_{n=m+1}^{\infty}}$ if and only if $\theta^{*}$ has the matrix
\[
	\begin{pmatrix}
		\theta^{*}_{1} & 0 \\
		0 & \theta^{*}_{4}
	\end{pmatrix}
\]
with respect to the decomposition
\begin{equation} \label{(2.2)}
	l_2(A) = L_{m}^{\perp} \dotplus U(L_{m}) \overset{\theta^{*}}{\longrightarrow} \theta^{*}(L_{m}^{\perp}) \oplus \bigl(\theta^{*}(L_{m}^{\perp})\bigr)^{\perp} = l_2(A) \, , 
\end{equation}
where $\theta^{*}_{1}$ and $U$ are isomorphisms. Now, clearly,
\[
	\theta^{*}(L_{m}^{\perp}) = \overline{{\rm Span}_{A} \left\{ x_n \right\}_{n=m+1}^{\infty}}
\]
since $\theta^{*}(L_{m}^{\perp})$ is closed, hence, by \eqref{(2.2)} and \cite[Lemma 3.1]{Ivkovic_2023} we deduce that $\theta^{*}$ is $A$-Fredholm if and only if $\left( \overline{{\rm Span}_{A} \left\{ x_n  \right\}_{n=m+1}^{\infty}} \right)^{\perp}$, which is equal to $\bigl(\theta^{*}(L_{m}^{\perp} \bigr)^{\perp}$, is finitely generated.
\end{proof}

It is known that the synthesis operator of Bessel sequences for Hilbert spaces has not necessarily to have a (norm-)closed image. At the same time classical Fredholm operators admit a closed image. Considering the Hilbert C*-module situation, generalized (lower semi-)$A$-Fredholm operators need not necessarily to have closed images by \cite[Ex.~6.0.13-15]{Ivkovic_2021}, and moreover, their norm-closure also might not be biorthogonally complemented anymore, e.g.~for injective multiplication operators like actions by admissible elements of the center of a C*-algebra with a non-isolated zero in their spectrum. 


\section{Adjointable semi-$A$-Weyl and semi-$A$-Weyl-Fredholm operators on Hilbert $A$-modules}

Following the theory of generalized Weyl operators on Hilbert spaces (cf.~\cite{Dj_2002}) we rely on one reasonable generalization to adjointable semi-$A$-Weyl operators on Hilbert $A$-modules as a subclass of semi-$A$-Fredholm operators. In the $A$-Fredholm case these operators are called $A$-Weyl-Fredholm operators. Basic results can be found in \cite{Ivkovic_2020_2,Ivkovic_2021,Ivkovic_2024_2}. The non-adjointable case is described in \cite{Ivkovic_2023}.

Reversing the point of view taken so far, we investigate the subclasses of Bessel sequences of Hilbert C*-modules with synthesis operators that are semi-$A$-Weyl or semi-$A$-Weyl-Fredholm operators. We take the standard Hilbert $A$-module $l_2(A)$ as a model. 

\begin{dfn} (cf.~\cite{Ivkovic_2020_2}.) \label{D D03}  \newline
Let $A$ be a unital C*-algebra. Let $F$ be an {\it upper semi-$A$-Fredholm operator on} $l_2(A)$.  We say that $F$ is an upper semi-$A$-Weyl operator on $l_2(A)$ if there exists a decomposition 
\[	
	l_2(A) = {\mathcal M}_1 \dotplus {\mathcal N}_1 \stackrel{F}{\longrightarrow} {\mathcal M}_2 \dotplus{\mathcal  N}_2= l_2(A)
\]	 
with respect to which
\begin{equation*}
		F=\left\lbrack
		\begin{array}{cc}
			F_{1} & 0 \\
			0 & F_{4} \\
		\end{array}
		\right \rbrack \, ,	
\end{equation*}
where $F_{1}$ is an isomorphism, ${\mathcal N}_1$ is closed, finitely generated and ${\mathcal N}_1 \preceq {\mathcal N}_2$, that is ${\mathcal N}_1$ is isomorphic to a closed submodule of ${\mathcal N}_2$. Consequently, the copy of ${\mathcal N}_1$ in ${\mathcal N}_2$ is a direct orthogonal summand.
	
Further, we define $F$ to be a {\it $A$-Weyl operator on} $l_2(A)$ if $F$ is a $A$-Fredholm operator for which there exists a matricial decomposition characteristic for $A$-Fredholm operators (cf.~Definition \ref{dfn_Fredholm})
\[
	l_2(A) ={\mathcal M}_1 \dotplus {\mathcal N}_1 \stackrel{F}{\longrightarrow} {\mathcal M}_2 \dotplus {\mathcal N}_2= l_2(A) \, ,
\]	
where ${\mathcal N}_1 \cong {\mathcal N}_2$. 
\end{dfn}

Semi-$A$-Weyl operators are an important subclass of semi-$A$-Fredholm operators. A natural question that arises in this connection is what kind of Bessel sequences are induced by semi-$A$-Weyl operators as their synthesis operators. To select a suitable subclass of Bessel sequences of Hilbert C*-modules we give the following definition.

\begin{dfn}
Let $A$ be a unital C*-algebra. Let $L_m$ be the $A$-submodule generated by the first $m$ orthonormal basis elements of $l_2(A)$. A {\it pseudo-Riesz-Weyl sequence} is a Bessel sequence $\{x_n\}_{n \in {\mathbb N}} \subseteq l_2(A)$ for which there exists a number $m \in {\mathbb N}$ such that $\{x_n\}_{m+1}^{\infty}$ is a modular Riesz basis for $ \overline{\operatorname{Span}_{A} \left\{ x_n \;\middle|\; m+1 \leq n,\; n \in {\mathbb N}\right\}}$ and
\begin{equation*}
	L_m \preceq \overline{\operatorname{Span}_{A} \left\{ x_n \;\middle|\; m+1 \leq n,\; n \in {\mathbb N} \right\}}^{\perp} \, ,
\end{equation*}
that is $L_m$ is isomorphic to a closed submodule of the orthogonal complement of
\begin{equation*}
	\overline{\operatorname{Span}_{A} \left\{ x_n \;\middle|\; m+1 \leq n,\; n \in {\mathbb N} \right\}} \, .
\end{equation*}
\end{dfn}

As we will see in the next proposition,  pseudo-Riesz-Weyl sequences in $l_2(A)$ are characterized by adjointable semi-$A$-Weyl operators as their synthesis operators.

\begin{prop} \label{semi-weyl}
Let $A$ be a unital C*-algebra. A Bessel sequence $\{x_n\}_{n \in {\mathcal N}}$ is a pseudo-Riesz-Weyl sequence for $l_2(A)$ if and only if the associated synthesis operator $\theta^{*}$ is adjointable and upper semi-$A$-Weyl.
\end{prop}

\begin{proof}
Suppose first that $\theta^{*}$ is adjointable and upper-semi-$\mathcal{A}$-Weyl, and let
\begin{equation*}
	l_2(A) = {\mathcal M}_{1} \dotplus {\mathcal N}_{1}  \overset{\theta^{*}}{\longrightarrow}
	{\mathcal M}_{2} \dotplus {\mathcal N}_{2} = l_2(A)
\end{equation*}
be a matricial decomposition for $\theta^{*}$ as indicated in Definition \ref{D D03}. Then, ${\mathcal N}_{1}$ is finitely generated and a direct orthogonal summand of $L_m$. Consequently, there exists some $m \in {\mathbb N}$ and some finitely generated closed submodule ${\mathcal P} \subseteq l_2(A)$ such that	
\begin{equation*}
	l_2(A) = L_{m}^{\perp} \dotplus ({\mathcal P} \dotplus {\mathcal N}_{1}) \overset{\theta^{*}}{\longrightarrow} \theta^{*}(L_{m}^{\perp}) \dotplus \big(\theta^{*}({\mathcal P}) \dotplus {\mathcal N}_{2}\big) = l_2(A)
\end{equation*}
is a matricial decomposition for $\theta^{*}$ as an adjointable and upper $A$-Fredholm operator, and $\theta^{*}$ realizes an isomorphism of $\mathcal P$ onto $\theta^{*}({\mathcal P})$. Since ${\mathcal N}_{1}\preceq {\mathcal N}_{2}$, it follows that
\begin{equation*}
	({\mathcal P} \dotplus {\mathcal N}_{1}) \preceq (\theta^{*}(P) \dotplus {\mathcal N}_{2}) \, .
\end{equation*}
Let $Q_{m}$ be the orthogonal projection onto $L_m^{\perp}$. Then $\operatorname{Im}(\theta^{*}Q_m)$, which is equal to $\theta^{*}(L_m^{\perp})$, 
is closed because $\theta^{*}$ is an isomorphism of $L_m^{\perp}$ onto its image $\theta^{*}(L_m^{\perp})$. Hence, $\theta^{*}(L_m^{\perp})$ is orthogonally complementable in $l_2(A)$ by \cite[Theorem 2.3.3]{MT}. Since
\begin{equation*}
	l_2(A) = \theta^{*}(L_m^{\perp}) \dotplus \bigl(\theta^{*}({\mathcal P}) \dotplus {\mathcal N}_{2}\bigr) = \theta^{*}(L_m^{\perp}) \oplus \bigl(\theta^{*}(L_m^{\perp})\bigr)^{\perp} \, , 
\end{equation*}
clearly,
\begin{equation*}
	\bigl(\theta^{*}(L_m^{\perp})\bigr)^{\perp} \cong \bigl(\theta^{*}(P) \dotplus {\mathcal N}_{2}\bigr) \, .
\end{equation*}
Similarly, since
\begin{equation*}
	l_2(A) = 	L_m^{\perp} \dotplus ({\mathcal P} \dotplus {\mathcal N}_{1}) = L_m^{\perp}\oplus L_m \, ,
\end{equation*}
we get
\begin{equation*}
	L_m 	\cong ({\mathcal P} \dotplus {\mathcal N}_{1}) \preceq \bigl(\theta^{*}(P) \dotplus {\mathcal N}_{2}\bigr) \cong \bigl(\theta^{*}(L_m^{\perp})\bigr)^{\perp} \, .
\end{equation*}
However, since $\theta^{*}(L_m^{\perp})$ is closed and $\theta^{*}(e_j)=x_j$ for all $j \in {\mathbb N}$ for the fixed orthonormal basis $\{ e_j \}_{j \in {\mathbb N}}$ of $l_2(A)$, obviously
\begin{equation*}
	\theta^{*}(L_m^{\perp}) = \overline{\operatorname{Span}_{A} \left\{ x_j \right\}_{j=m+1}^{\infty}} \, ,
\end{equation*}
so
\begin{equation*}
	L_m 	\preceq 	\overline{\operatorname{Span}_{A} \left\{ x_j \right\}_{j=m+1}^{\infty}}^{\perp} \, .
\end{equation*}
Thus, $\{x_n\}_{n \in {\mathbb N}}$ is a pseudo-Riesz-Weyl sequence.

Conversely, suppose that $\{x_n\}_{n \in {\mathbb N}}$ is a pseudo-Riesz-Weyl sequence. Let $m \in {\mathbb N}$ be such that $\{x_n\}_{n=m+1}^{\infty}$ is a modular Riesz basis for  $ \overline{\operatorname{Span}_{A} \left\{ x_n \;\middle|\; m+1 \leq n,\; n \in {\mathbb N} \right\}}$ and
\begin{equation*}
	L_m 	\preceq 	\overline{\operatorname{Span}_{A} \left\{ x_n \right\}_{n=m+1}^{\infty}}^{\perp} \, .
\end{equation*}
Since $\{x_n\}_{n=m+1}^{\infty}$ is a modular Riesz basis, there exists an adjointable bounded $A$-linear operator
\begin{equation*}
	\theta_{1} : l_2(A)  \longrightarrow \overline{\operatorname{Span}_{A} \left\{ x_n \right\}_{n=m+1}^{\infty}}
\end{equation*}
such that $\theta_{1}$ is an isomorphism and $\theta_{1}(e_j)=x_{j+m}$ for all $j \in {\mathbb N}$. However, then by $\theta_{1} = \theta^{*}F^{m}$, 
where $F$ denotes the forward shift, so we have that $	\theta^{*}$ realizes an isomorphism of $L_m^{\perp}$ and its image. In particular, $\theta^{*}(L_m^{\perp})$ is closed, hence, by the same arguments as above, applying \cite[Theorem 2.3.3]{MT}, we get that $\theta^{*}(L_m^{\perp})$ is orthogonally complementable in $l_2(A)$, so
\begin{equation*}
	l_2(A) = \theta^{*}(L_m^{\perp}) \oplus \bigl(\theta^{*}(L_m^{\perp})\bigr)^{\perp} \, .
\end{equation*}
However,
\begin{equation*}
	\bigl(\theta^{*}(L_m^{\perp})\bigr)^{\perp} = 	\overline{\operatorname{Span}_{A} \left\{ x_n \right\}_{n=m+1}^{\infty}}^{\perp} \, ,
\end{equation*}
so
\begin{equation*}
	L_m 	\preceq \bigl(\theta^{*}(L_m^{\perp})\bigr)^{\perp} \, .
\end{equation*}
Since $\theta^{*}$ is an isomorphism of $L_m^{\perp}$ and $\theta^{*}(L_m^{\perp})$, it follows that $\theta^{*}$ has the matrix representation
\begin{equation*}
\begin{pmatrix}
		(\theta^{*})_{1} & (\theta^{*})_{2} \\
		0 & (\theta^{*})_{4}
\end{pmatrix}
\end{equation*}
with respect to the decomposition
\begin{equation*}
	l_2(A) = 	L_m^{\perp}\oplus L_m \overset{\theta^{*}}{\longrightarrow} \theta^{*}(L_m^{\perp}) \oplus \bigl(\theta^{*}(L_m^{\perp}) \bigr)^{\perp} = l_2(A) \, ,
\end{equation*}
where $(\theta^{*})_{1}$ is an isomorphism. By the same ``diagonalization'' method as in the proof of 	\cite[Lemma 2.7.10]{MT}, there exists an isomorphism
$U$ of $l_2(A)$ such that $\theta^{*}$ has the matrix representation
\begin{equation*}
	\begin{pmatrix}
		(\theta^{*})_{1} & 0 \\
		0 & (\widetilde{\theta}^{*})_{4}
	\end{pmatrix}
\end{equation*}
with respect to the decomposition
\begin{equation*}
	l_2(A) = L_m^{\perp} \dotplus {\mathcal U}(L_m) \overset{\theta^{*}}{\longrightarrow} \theta^{*}(L_m^{\perp}) \oplus \bigl(\theta^{*}(L_m^{\perp})\bigr)^{\perp} = l_2(A)
\end{equation*}
where $(\theta^{*})_{1}$ is an isomorphism. Since
\begin{equation*}
	U(L_m) \cong L_m \preceq 	\bigl(\theta^{*}(L_m^{\perp})\bigr)^{\perp} \, ,
\end{equation*}
we deduce that the synthesis operator $\theta^{*}$ is an adjointable and upper-semi-$A$-Weyl operator.
\end{proof}

\begin{dfn}
Let $A$ be a unital C*-algebra and $l_2(A)$ be be the standard countably generated Hilbert $A$-module. Let $L_m$ be the $A$-submodule generated by the first $m$ orthonormal basis elements of $l_2(A)$.	A sequence $\{x_n\}_{n \in {\mathbb N}} \subseteq l_2(A)$ is called {\it pseudo-Riesz-Weyl basis} if there exists some $m \in {\mathbb N}$ such that $\{x_n\}_{n=m+1}^{\infty}$ is a modular Riesz basis for $\overline{\operatorname{Span}_{A} \left\{ x_n \right\}_{n=m+1}^{\infty}}$ and $L_m \cong \overline{\operatorname{Span}_{A} \left\{ x_n \right\}_{n=m+1}^{\infty}}^{\perp}$.
\end{dfn}

\begin{rmk}
Let $A={\mathbb C}$ and $\{ x_n \}_{n \in \mathbb N} \in l_2$ be a pseudo-Riesz basis. Suppose, $m$ is the smallest number of elements of $\{ x_n \}_{n \in \mathbb N}$ that need to be removed such that the remaining sequence becomes a Riesz basis for the closure of its linear span, and let $j_1, ...,j_m$ be the indicees of these elements. If $\{ x_n \}_{n \in \mathbb N}$ is additionally a pseudo-Riesz-Weyl basis, then  $\dim (\overline{{\rm Span} \left\{ x_n \}_{n \in {\mathbb N} \setminus \{ j_1, ... , j_m \}} \right\}})^\perp = m$, where the orthogonal complement is taken in $l_2$. The reverse implication also holds. 
\end{rmk}

Below, we show a particular property of pseudo-Riesz-Weyl bases which motivated us to introduce this special type of Bessel sequences in $l_2(A)$. 

\begin{prop}
Let $A$ be a unital C*-algebra and $\{ x_n \}_{n \in \mathbb N} \in l_2(A)$ be a pseudo-Riesz-Weyl basis. Set 
\[
M:= \inf \{ m \in {\mathbb N} : \{ x_n \}_{n=m+1}^\infty \,\, {\rm is} \, {\rm a} \, {\rm modular} \, {\rm Riesz} \, {\rm basis} \, {\rm for} \,\, \overline{{\rm Span}_{A} \{ x_n \}_{n=m+1}^\infty} \} \, .
\]
Let $m \geq M$ be any integer. Then the $A$-submodule $\overline{{\rm Span}_{A} \{ x_n \}_{n=m+1}^\infty}$ is orthogonally complementable in $l_2(A)$ for any $m \geq M$, and the classes of $L_m$ and of $(\overline{{\rm Span}_{A} \{ x_n \}_{n=m+1}^\infty})^\perp$ in the $K$-group $K_0(A)$ coincide. In particular, whenever ${\mathcal N}_1$ and ${\mathcal N}_2$ are algebraically finitely generated $A$-submodules of $l_2(A)$, and the classes of ${\mathcal N}_1$ and of ${\mathcal N}_2$ coincide in $K_0(A)$ if and only if ${\mathcal N}_1 \cong {\mathcal N}_2$, then $(\overline{{\rm Span}_{A} \{ x_n \}_{n=m+1}^\infty})^\perp$ has $m$ generators.
\end{prop}

\begin{proof}
Fix $m \geq M$. Let $\theta^*$ be the synthesis operator of $\{ x_n \}_{n \in \mathbb N}$. By the proof of Proposition \ref{semi-weyl}, $\theta^*$ has a matricial decomposition
\begin{equation*}
	\begin{pmatrix}
		(\theta^{*})_{1} & 0 \\
		0 & (\widetilde{\theta}^{*})_{4}
	\end{pmatrix}
\end{equation*}
with respect to the decomposition 
\begin{equation*}
	l_2(A) = 	L_m^{\perp}\dotplus U(L_m)  {\longrightarrow} \theta^{*}  (L_m^{\perp}) \oplus \bigl(\theta^{*}(L_m^{\perp}) \bigr)^{\perp} = l_2(A) \, ,
\end{equation*}
where $(\theta^{*})_{1}$ is an isomorphism. Since $\theta^*$ is $A$-Weyl it is $A$-Fredholm, in particular. Consequently, by \cite[Lemma 3.1.8]{Ivkovic_2021} the $A$-submodule $(\theta^*(L_M^\perp))^\perp$ has to be algebraically finitely generated, and thus, $0 = {\rm index}(\theta^*) = [ L_m ] - [\theta^*(L_m^\perp)^\perp]$. Since $\theta^*(L_m^\perp) =  \overline{{\rm Span}_{A} \{ x_n \}_{n=m+1}^\infty}$ the claim follows.
\end{proof}

 We can show the following interrelation in a very similar way as in the proof of Proposition \ref{semi-weyl}.

\begin{prop} \label{weyl}
Let $A$ be a unital C*-algebra. A sequence $\{x_n\}_{n=1}^{\infty}$ in $l_2(A)$  is a pseudo-Riesz-Weyl-basis if and only if the associated synthesis operator is $A$-Weyl.
\end{prop}

We want to show the advantage of pseudo-Riesz-Weyl bases over only pseudo-Riesz bases to emphasize their good properties.

\begin{prop} \label{pseudo-RW_char}
Let $A$ be a unital C*-algebra and $\{ x_n \}_{n \in \mathbb N} \in l_2(A)$ be a Bessel sequence. Then the following assertions hold:

(1) Suppose $\{ x_n \}_{n \in \mathbb N}$ is a pseudo-Riesz basis. Let $M \in \mathbb N$ be the smallest number $m \in \mathbb N$ such that $\{ x_n \}_{n=m+1}^\infty$ is a modular Riesz basis for $\overline{\operatorname{Span}_{A} \left\{ y_n \right\}_{n=m+1}^{\infty}}$, and, at the same time, such that $\overline{\operatorname{Span}_{A} \left\{ y_n \right\}_{n=m+1}^{\infty}}^\perp$ is finitely generated. Then for every $m \geq M$, the sequence $\{ x_n \}_{n=m+1}^\infty$ is a modular Riesz basis for $\overline{\operatorname{Span}_{A} \left\{ y_n \right\}_{n=m+1}^{\infty}}$ which is orthogonally complementable in $l_2(A)$, and $\overline{\operatorname{Span}_{A} \left\{ y_n \right\}_{n=m+1}^{\infty}}^\perp$ is finitely generated. 

(2)  Suppose $\{ x_n \}_{n \in \mathbb N}$ is a pseudo-Riesz-Weyl basis. Let $M \in \mathbb N$ be the smallest number $m \in \mathbb N$ such that $\{ x_n \}_{n=m+1}^\infty$ is a modular Riesz basis for $\overline{\operatorname{Span}_{A} \left\{ y_n \right\}_{n=m+1}^{\infty}}$, and, at the same time, such that $L_m \cong \overline{\operatorname{Span}_{A} \left\{ y_n \right\}_{n=m+1}^{\infty}}^\perp$. Then for every 
\linebreak[4]
$m \geq M$, the sequence $\{ x_n \}_{n=m+1}^\infty$ is a modular Riesz basis for $\overline{\operatorname{Span}_{A} \left\{ y_n \right\}_{n=m+1}^{\infty}}$ which is orthogonally complementable in $l_2(A)$, and $L_m \cong \overline{\operatorname{Span}_{A} \left\{ y_n \right\}_{n=m+1}^{\infty}}^\perp$.
\end{prop}

\begin{proof}
    Let $\theta^*$ be the synthesis operator of $\{ x_n \}_{n \in \mathbb N}$ and $m \geq M$. Note that $\theta^*_{|_{L_{M}^{\perp}}}$ is an isomorphism onto its image. Then $\theta^*_{|_{L_{m}^{\perp}}}$ with $m \geq M$ is also an isomorphism onto its image. If $P_m$ denotes the orthogonal projection onto $L_{m}^{\perp}$, then the operator $\theta^*P_m$ is adjointable with closed image $\theta^*P_m(L_M^\perp) = \theta^*(L_m^\perp)$. By \cite[Thm.~2.3.3]{MT} the Hilbert $A$-submodule $\theta^*(L_m^\perp)$ is orthogonally complementable. We obtain the decompositions
\begin{eqnarray*}
    l_2(A) & = & \theta^*(L_m^\perp) \oplus (\theta^*(L_m^\perp))^\perp \\
              & = & \theta^*(L_m^\perp)  \dotplus \theta^*(L_M^\perp \cap L_m) \dotplus (\theta^*(L_M^\perp)^\perp \, .
\end{eqnarray*}
We use the identity $\theta^*(L_M^\perp) = \theta^*(L_m^\perp) \dotplus \theta^*(L_M^\perp \cap L_m)$ because $L_M^\perp = L_m^\perp \oplus (L_M^\perp \cap L_m)$ and  $\theta^*_{|_{L_{M}^{\perp}}}$  is an isomorphism what turns a simple set sum into a direct topological sum. Moreover, we use the property of $\theta^*(L_M^\perp)$ to be orthogonally complementable in $l_2(A)$. Consequently,
\begin{equation*}
    \theta^*(L_m^\perp) \cong \theta^*(L_M^\perp \cap L_m) \dotplus (\theta^*(L_M^\perp))^\perp         \, .
\end{equation*}
    Since both the Hilbert $A$-submodules $\theta^*(L_M^\perp \cap L_m)$ and $(\theta^*(L_M^\perp))^\perp$ are finitely generated, so is their direct sum , and hence $(\theta^*(L_m^\perp))^\perp$. This proves assertion (1).
    
     If additionally $\{ x_n \}_{n \in \mathbb N}$ is a pseudo-Riesz-Weyl basis and $(\theta^*(L_M^\perp))^\perp \cong L_M$ then we deduce the isomorphism of Hilbert $A$-submodules $(\theta^*(L_m^\perp))^\perp \cong (L_M^\perp \cap L_m) \oplus L_M = L_m$. This proves assertion (2).
\end{proof}

\begin{rmk}
The key advantage of pseudo-Riesz-Weyl bases over general pseudo-Riesz bases is the explicit description of $\overline{\operatorname{Span}_{A} \left\{ y_n \right\}_{n=m+1}^{\infty}}^\perp$ for every $m \geq M$ as a free finitely generated $A$-module. Depending on the structure of $A$ the set of algebraically finitely generated Hilbert $A$-modules is most often quite diverse with respect to isomorphisms, and so the set of pseudo-Riesz, non-pseudo-Riesz-Weyl bases of $l_2(A)$.

In the next section, we want to demonstrate that every sufficiently small perturbation of a pseudo-Riesz-Weyl basis is a pseudo-Riesz-Weyl basis again. This shows the set of all pseudo-Riesz-Weyl bases in $l_2(A)$ to be a very large set.
\end{rmk}


\section{Stability of pseudo-Riesz and pseudo-Riesz-Weyl sequences and bases under perturbations}

We are interested in the stability of the classes of pseudo-Riesz sequences or bases, as well as of pseudo-Riesz-Weyl sequences or bases under classical perturbations. Again, we resort to the standard Hilbert $A$-module $l_2(A)$. The point of view goes back to the Paley-Wiener theorem for Bessel sequences, frames and bases of Hilbert spaces, cf.~\cite[Thm. 5.24, Exerc.~8.10]{Heil}, \cite[Thm. 22.1.1]{Chr} (cf.~Thm.~\ref{thm-Bessel} and Prop.~\ref{prop_sequ}.) While the theory of perturbations of frames of Hilbert spaces is well-known (\cite[pp. 557-575]{Chr} there are only a few results about perturbations of module frames and Riesz bases in the theory of Hilbert C*-modules. Major results were obtained by Deguang Han, Wu Jing and R.~N.~Mohapatra in \cite[Sect.~3]{HJM_2009} in 2009, with reference to the 2006 Ph.D.~thesis by Wu Jing \cite{Jing_2006}. Lj.~Aramba{\v{s}}i\'c and D.~Baki\'c investigated perturbation results for inner and outer module frames in the situation of Hilbert C*-modules over non-unital C*-algebras involving multiplier modules, cf.~\cite{AB_2017} in 2017. Further results on the perturbation of module Bessel sequences and of module continuous frames were obtained by A.~Karara, K.~Mabrouk and M.~Rossafi in \cite{KMR_2026} in 2026 and by H.~Ghasemi, T.~L.~Shateri in \cite{GS_2024} in 2024, respectively.

The approach taken here makes use of the discovered correspondence between the named classes of module Bessel sequences of Hilbert C*-modules and of certain classes of (upper semi-)$A$-Fredholm operators describing their synthesis operators, varying the latter suitably and interpreting the results in terms of the respective classes of Bessel sequences under consideration. Such kind of reasonings have been made by D.~Biswas and M.~Mitkovski in \cite{BM_2025} for near-Riesz and pseudo-Riesz bases of Hilbert spaces. Our results are new even for those Bessel sequences of Hilbert spaces which correspond to classical semi-Weyl operators as their synthesis operators, cf.~\cite{Dj_2002}. 

The starting point are perturbations of pseudo-Riesz sequences and pseudo-Riesz bases of Hilbert C*-modules.

\begin{prop}   \label{perturbacija}
        Let $A$ be a unital C*-algebra. Consider a pseudo-Riesz sequence $\{x_n\}_{n \in {\mathbb N}} \in l_2(A)$ and a number $N \in {\mathbb N}$ such that $\{x_n\}_{n=N+1}^{\infty}$ is a modular Riesz basis for $ \overline{\operatorname{Span}_{A} \left\{ x_n \right\}_{n=N+1}^{\infty}}$. Suppose that the number $m>0$ is a lower frame bound for $\{x_n\}_{n=N+1}^{\infty}$, i.e. 
\begin{equation*}
	m\langle \widetilde{x},\widetilde{x}\rangle \leq \sum_{n=N+1}^{\infty} \langle \widetilde{x},x_n\rangle \langle x_n,\widetilde{x} \rangle
\end{equation*}
for any $\widetilde{x} \in \overline{\operatorname{Span}_{A} \left\{ x_n \right\}_{n=N+1}^{\infty}}$. Let $\{y_n\}_{n \in {\mathbb N}}$ be another Bessel sequence in $l_2(A)$ such that
\begin{equation*}
	\sum_{n=N+1}^{\infty} 	\langle x,x_n-y_n \rangle \langle x_n-y_n,x \rangle \leq m \langle x,x\rangle
\end{equation*}
for any $x \in l_2(A)$. Then $\{y_n\}_{n \in {\mathbb N}}$ is also a pseudo-Riesz sequence. Moreover, if $\{x_n\}_{n \in {\mathbb N}}$ is \underline{not} a pseudo-Riesz-Weyl-sequence, then $\{y_n\}_{n \in {\mathbb N}}$ is \underline{not} a pseudo-Riesz-Weyl sequence.
\end{prop}	

\begin{proof}
Denote the $A$-submodule generated by the first $m$ orthonormal basis elements of $l_2(A)$ by $L_m$. Let $\psi$ be the synthesis operator of $\{x_n\}_{n=N+1}^{\infty}$ and $\theta^*$ be the synthesis operator of $\{ x_n \}_{n=1}^\infty$. Then
\begin{equation*}
	{\theta^*_{|_{L_{N}^{\perp}}}} = \psi B_{|_{L_{N}^{\perp}}}^{N}\, ,
\end{equation*}
where $B$ denotes the backward shift (since $\psi(e_j)=f_{j+N}$ for all $j \in {\mathbb N}$). Since the norm of the synthesis operator is bounded from below by the square root of the lower frame bound for any standard Bessel sequence of $l_2(A)$, we have 
\begin{equation*}
	\|\psi (y)\| \geq \sqrt{m}\,\|y\|
\end{equation*}
for any $y \in l_2(A)$, hence
\begin{equation*}
	\|\theta^*(x)\| \geq \sqrt{m}\,\|x\|
\end{equation*}	
for any $x \in L_{N}^{\perp}$ since $B_{|_{L_{N}^{\perp}}}^{N}$ is an isometry. As in the proof of Proposition \ref{semi-weyl}, $\theta^*_{|_{L_{N}^{\perp}}}$  is an isomorphism onto its image, and $\theta^*(L_{N}^{\perp})$ is orthogonally complementable in $l_2(A)$. Let $Q$ be the orthogonal projection onto $\theta^*(L_{N}^{\perp})$ and $(\theta^*)_{1} :=  \theta^*_{|_{L_{N}^{\perp}}} = 	Q\,\theta^*_{|_{L_{N}^{\perp}}}$. We have
\begin{equation*}
	\left\| \big((\theta^*)_{1}\big)^{-1} 	\right\| \leq \frac{1}{\sqrt{m}} \, ,
\end{equation*}	
hence
\begin{equation*}
	\frac{1} {\left\| \big((\theta^*)_{1}\big)^{-1} \right\|} \geq \sqrt{m} \, .
\end{equation*}
Therefore, if a bounded adjointable operator $D : L_{N}^{\perp} \longrightarrow \theta(L_{N}^{\perp})$ satisfies the inequality
\begin{equation*}
	\|D-\theta^*_{1}\| < \sqrt{m} \, ,
\end{equation*}
then the operator $D$ is also an isomorphism. Let $\varphi$ be the synthesis operator of $\{y_n\}_{n=1}^{\infty}$. Since $\{x_n-y_n\}_{N+1}^\infty$ is supposed to be a Bessel sequence with an upper bound $\sqrt{m}$, for the respective difference of synthesis operators we get
\begin{equation*}
	\| (\theta^{*}-\varphi)_{|_{L_{N}^{\perp}}} \| < \sqrt{m} \, .
\end{equation*}
Hence we get
\begin{equation*}
	\left\| (\theta^{*})_{1}  - Q\varphi_{|_{L_{N}^{\perp}}} \right\| = \left\| Q(\theta^{*}-\varphi)_{|_{L_{N}^{\perp}}} \right\| \leq \| (\theta^{*} - \varphi)_{|_{L_{N}^{\perp}}} \| < \sqrt{m} \, ,
\end{equation*}
so $Q\varphi^*_{|_{L_{N}^{\perp}}}$ is also an isomorphism from $L_{N}^{\perp}$ onto $\theta^{*}(L_{N}^{\perp})$. Thus, $\varphi$ has the matrix
\begin{equation*}
	\begin{pmatrix}
		\varphi_{1} & \varphi_{2} \\
		\varphi_{3} & \varphi_{4}
	\end{pmatrix}
\end{equation*}
with respect to the decomposition
\begin{equation*}
	l_2(A) = 	L_{N}^{\perp}\oplus L_{N} \overset{\varphi}{\longrightarrow} \theta^{*}(L_{N}^{\perp}) \oplus \bigl(\theta^{*}(L_{N}^{\perp})\bigr)^{\perp} = l_2(A) \, ,
\end{equation*}
where $\varphi_{1}$ is an isomorphism. By the same ``diagonalization'' method as in the proof of 	\cite[Lemma 2.7.10]{MT} we can find 	isomorphisms $U$ and $V$ of $l_2(A)$ such that $\varphi$ has the matricial decomposition
\begin{equation*}
	\begin{pmatrix}
		\widetilde{\varphi}_{1} & 0 \\
		0 & \widetilde{\varphi}_{4}
	\end{pmatrix}
\end{equation*}	
with respect to the Hilbert $A$-module topological sum decomposition 	
\begin{equation*}
	l_2(A) = 	L_{N}^{\perp}\dotplus U(L_{N}) \overset{\varphi}{\longrightarrow} V \bigl( \theta^{*}(L_{N}^{\perp}) \bigr) \dotplus \bigl(\theta^{*}(L_{N}^{\perp})\bigr)^{\perp} = l_2(A) \, ,
\end{equation*}
where $\varphi_{1}$ is an isomorphism, so we have obtained an upper semi-$A$-Fredholm decomposition for $\varphi$. 	Thus, $\varphi$ is an upper semi-$A$-Fredholm operator, and therefore, $\{y_n\}_{n \in {\mathbb N}}$ is a pseudo-Riesz sequence, which proves the first statement.

Now, observe also that $\theta^{*}$ has the matrix	
\begin{equation*}
	\begin{pmatrix}
		(\theta^{*})_{1} & (\theta^{*})_{2} \\
		0 & (\theta^{*})_{4}
	\end{pmatrix}
\end{equation*}
with respect to the topological direct sum decomposition	
\begin{equation*}
	l_2(A) = L_{N}^{\perp}\oplus L_{N} \overset{\theta^{*}}{\longrightarrow} \theta^{*}(L_{N}^{\perp}) \oplus \bigl(\theta^{*}(L_{N}^{\perp} \bigr)^{\perp} = l_2(A) \, ,
\end{equation*}
(where $(\theta^{*})_{1}$ is an isomorphism). Hence, again by the same method as in the proof of \cite[Lemma 2.7.10]{MT} there is an isomorphism $\widetilde{U}$ of $l_2(A)$ such that
\begin{equation*}
	l_2(A) = L_{N}^{\perp} \dotplus \widetilde{U}(L_{N}) \overset{\theta^{*}}{\longrightarrow}	\theta^{*}(L_{N}^{\perp})	\oplus \bigl(\theta^{*}(L_{N}^{\perp})\bigr)^{\perp} = l_2(A)
\end{equation*}
is an upper semi-$A$-Riesz decomposition for $\theta^{*}$. Thus, we constructed a an upper semi-Fredholm decomposition for
\begin{equation*}
	l_2(A) = 	{\mathcal M}_{1} \dotplus {\mathcal N}_{1} \overset{\theta^{*}}{\longrightarrow} {\mathcal M}_{2}\dotplus {\mathcal N}_{2} = l_2(A)
\end{equation*}
for $\theta^{*}$ and a module decomposition
\begin{equation*}
	l_2(A) = {\mathcal M}_{1}^{\prime} \dotplus {\mathcal N}_{1}^{\prime} \overset{\varphi}{\longrightarrow} {\mathcal M}_{2}^{\prime} \dotplus {\mathcal N}_{2}^{\prime} = l_2(A)
\end{equation*}
for $\varphi$, respectively, with
\begin{equation*}
	{\mathcal M}_{1} = {\mathcal M}_{1}^{\prime}=L_{N}^{\perp} \, , \quad
	{\mathcal N}_{2} = {\mathcal N}_{2}^{\prime} = \bigl(\theta^{*}(L_{N}^{\perp})\bigr)^{\perp} \, ,
\end{equation*}
\begin{equation*}
	{\mathcal N}_{1} = 	\widetilde{U}(L_{N}) \cong 	L_{N} \cong U(L_{N}) = {\mathcal N}_{1}^{\prime} \, ,
\end{equation*}
and	
\begin{equation*}
	{\mathcal M}_{2} = \theta^{*}(L_{N}^{\perp}) \cong 	V\bigl(\theta^{*}(L_{N}^{\perp})\bigr) = {\mathcal M}_{2}^{\prime} \, .
\end{equation*}
If $\{x_n\}_{n=1}^{\infty}$ is \underline{not} a pseudo-Riesz-Weyl-sequence, then by Proposition \ref{semi-weyl}, $\theta^{*}$ is \underline{not} an upper semi-$A$-Weyl operator. Suppose at the contrary that $\{y_n\}_{n=1}^{\infty}$ is actually a pseudo-Riesz-Weyl-sequence. Then, by Proposition \ref{semi-weyl}, $\varphi$ must be an upper-semi-$A$-Weyl operator. However, since there exist upper semi-$A$-Fredholm decompositions for
\begin{equation*}
	l_2(A) = {\mathcal M}_{1} \dotplus {\mathcal N}_{1} 	\overset{\theta^{*}}{\longrightarrow} {\mathcal M}_{2} \dotplus {\mathcal N}_{2} =  l_2(A)
\end{equation*}
and	
\begin{equation*}
	l_2(A) = {\mathcal M}_{1}^{\prime} \dotplus {\mathcal N}_{1}^{\prime} \overset{\varphi}{\longrightarrow} {\mathcal M}_{2}^{\prime} \dotplus {\mathcal N}_{2}^{\prime} = l_2(A)
\end{equation*}
for $\theta^{*}$ and $\varphi$,	respectively, and
\begin{equation*}
	{\mathcal M}_{1}\cong {\mathcal M}_{1}^{\prime},
	\quad
	{\mathcal N}_{1}\cong {\mathcal N}_{1}^{\prime},
	\quad
	{\mathcal M}_{2}\cong {\mathcal M}_{2}^{\prime},
	\quad
	{\mathcal N}_{2}\cong {\mathcal N}_{2}^{\prime}.
\end{equation*}
if $\varphi$ was really an upper-semi-$A$-Weyl operator, then there would exist an upper semi-$A$-Fredholm decomposition 
\begin{equation*}
	l_2(A) = {\mathcal M}_{1}^{\prime \prime} \dotplus {\mathcal N}_{1}^{\prime \prime}  \overset{\varphi}{\longrightarrow} {\mathcal M}_{2}^{\prime \prime } \dotplus {\mathcal N}_{2}^{\prime \prime} = l_2(A)
\end{equation*}
for $ \varphi $. By letting $ \theta^{*} $ and $\varphi$ play the role of operators $F$ and $D$, respectively, in the proof of \cite[Theorem 4.2]{Ivkovic_2024_1}, 
we would be able to construct a semi-$A$-Weyl decomposition for $\theta^{*}$, what is a contradiction because $\theta^{*}$ is not an upper-semi-$A$-Weyl operator. Thus, we conclude that $\varphi$ is \underline{not} an upper-semi-$\mathcal{A}$-Weyl operator, so $\{y_n\}_{n=1}^{\infty}$ is \underline{not} a pseudo-Riesz sequence. This proves the second statement.
\end{proof}

\begin{prop} \label{prop_perturb1}
The following assertions hold under the notation and suppositions of Proposition \ref{perturbacija}:
	
	(1) Let $\{f_n\}_{n=1}^{\infty}$ be a pseudo-Riesz-Weyl-sequence in $l_2(A)$ and $N \in {\mathbb N}$ be such that $\{x_n\}_{n=N+1}^{\infty}$ is a modular Riesz-basis for $\overline{\operatorname{Span}_{A} \left\{ x_n  \right\}_{n=N+1}^{\infty}}^{\perp}$ and 
\begin{equation*}
	L_N 	\preceq \overline{\operatorname{Span}_{A} \left\{ x_n \right\}_{n=N+1}^{\infty}}^{\perp} \, .
\end{equation*}
Suppose that $m>0$ is a lower frame bound for $\{x_n\}_{n=N+1}^{\infty}$. Let $\{y_n\}_{n=1}^{\infty}$ be another Bessel sequence in $l_2(A)$ such that
\begin{equation*}
	\sum_{n=N+1}^{\infty}	\langle x,x_n-y_n \rangle \langle x_n-y_n,x \rangle \leq m \langle x,x\rangle
\end{equation*}
for any $x \in l_2(A)$. Then $\{y_n\}_{n=1}^{\infty}$ is also a pseudo-Riesz-Weyl sequence in $l_2(A)$. 	
	
	(2) Let $\{x_n\}_{n=1}^{\infty}$ be a pseudo-Riesz-Weyl basis in 	$l_2(A)$ and $N \in {\mathbb N}$ be such that $\{x_n\}_{n=N+1}^{\infty}$ is a modular Riesz basis for $	\overline{\operatorname{Span}_{A} \left\{ x_n \right\}_{n=N+1}^{\infty}}^{\perp}$, and
\begin{equation*}
	L_N 	\cong \overline{\operatorname{Span}_{A} \left\{ x_n \right\}_{n=N+1}^{\infty}}^{\perp} \, .
\end{equation*}
Suppose that $m>0$ is a lower frame bound for $\{x_n\}_{n=N+1}^{\infty}$, and let $\{y_n\}_{n=1}^{\infty}$ be a Bessel sequence in $l_2(A)$ such that
\begin{equation*}
	\sum_{n=N+1}^{\infty} \langle x,x_n-y_n\rangle \langle x_n-y_n,x\rangle \leq m \langle x,x\rangle
\end{equation*}
for any $x \in l_2(A)$. Then $\{y_n\}_{n=1}^{\infty}$ is also a pseudo-Riesz-Weyl-basis in $l_2(A)$. Moreover,  $\overline{\operatorname{Span}_{A} \left\{ y_n \right\}_{n=k+1}^{\infty}}$ admits $\{ y_n \}_{n=k+1}^\infty$ as a modular Riesz basis for every $k \geq N$, and $ \overline{\operatorname{Span}_{A} \left\{ x_n \right\}_{n=k+1}^{\infty}}^{\perp} \cong L_k$ for every $k \geq N$.
\end{prop}

\begin{proof}
	Let $\theta^{*}$ and $\varphi$ be synthesis operators	for $\{x_n\}_{n \in {\mathbb N}}$ and $\{y_n\}_{n \in {\mathbb N}}$, respectively. Proceeding as in the proof of Proposition \ref{perturbacija} we deduce that $\varphi$ has the matrix
\begin{equation*}
	\begin{pmatrix}
		\widetilde{\varphi}_{1} & 0 \\
		0 & \widetilde{\varphi}_{4}
	\end{pmatrix}
\end{equation*}
with respect to the decomposition
\begin{equation*}
	l_2(A) = L_{N}^{\perp} \dotplus U(L_{N}) \overset{\varphi}{\longrightarrow} V\bigl(\theta^{*}(L_{N}^{\perp})\bigr) \dotplus \bigl(\theta^{*}(L_{N}^{\perp})\bigr)^{\perp} = l_2(A) \, ,
\end{equation*}
where $U, V, \widetilde{\varphi}_{1}$ are isomorphisms. Hence, if $L_{N} \preceq 	\bigl(\theta^{*}(L_{N}^{\perp})\bigr)^{\perp}$, 	then $U(L_{N})$ and 
and $\bigl(\theta^{*}(L_{N}^{\perp})\bigr)^{\perp}$ are closed and finitely generated, and $U(L_{N}) \preceq \bigl(\theta^{*}(L_{N}^{\perp})\bigr)^{\perp}$, so  $\{y_n\}_{n=1}^{\infty}$ is also a pseudo-Riesz-Weyl sequence in $l_2(A)$. 

Furthermore, in the case of $L_{N} \cong \bigl(\theta^{*}(L_{N}^{\perp})\bigr)^{\perp}$, then $\varphi$ is $A$-Fredholm with  $U(L_{N}) \cong \bigl(\theta^{*}(L_{N}^{\perp})\bigr)^{\perp}$, hence $\{y_n\}_{n \in {\mathbb N}}$ is a pseudo-Riesz-Weyl-basis.

The analogous statement follows for any $k \geq N$ from Proposition \ref{pseudo-RW_char}, (2).
\end{proof}

There are many pseudo-Riesz-Weyl bases in $l_2(A)$. Let $\{ x_n \}_{n \in \mathbb N}$ be a pseudo-Riesz-Weyl basis, the integers $N$ and $m$ be as in Proposition \ref{prop_perturb2},(2), and $\{ e_n \}_{n \in \mathbb N}$ be an orthonormal Hilbert basis of $l_2(A)$. 
If $T$ is any adjointable bounded $A$-linear operator on $l_2(A)$ with a norm of $T$ restricted to $L_N^\perp$ less than $m^{1/2}$, then the sequence $\{ x_n+T(e_n) \}_{n \in \mathbb N}$ is also a pseudo-Riesz-Weyl basis. Moreover, the other properties described in \ref{prop_perturb2},(2) hold, too. 
Therefore, one initial pseudo-Riesz-Weyl basis gives rise to a large set of new pseudo-Riesz-Weyl bases with similar properties.

\begin{prop} \label{prop_perturb2}
Let $A$ be a unital C*-algebra. Let $\{x_n\}_{n \in {\mathbb N}}$ be a pseudo-Riesz sequence of $l_2(A)$ and $N \in {\mathbb N}$ be such that $\{x_n\}_{n=N+1}^{\infty}$ is a modular Riesz basis for $\overline{\operatorname{Span}_{A} \left\{ x_n \right\}_{n=N+1}^{\infty}}$.
Suppose that $m$ is a lower frame bound for $\{x_n\}_{n=N+1}^{\infty}$. Let $\{y_n\}_{n=1}^{\infty}$ be another Bessel sequence in $l_2(A)$ such that
\begin{equation*}
	\sum_{n=N+1}^{\infty} 	\langle x,x_n-y_n \rangle 	\langle x_n-y_n,x \rangle \leq m \langle x,x\rangle
\end{equation*}
for any $x \in l_2(A)$.
	
	(1) If $\{x_n\}_{n \in {\mathbb N}}$ is not a pseudo-Riesz basis,
	then $\{y_n\}_{n \in {\mathbb N}}$ is not a pseudo-Riesz basis.
	
	(2) If $\{x_n\}_{n \in {\mathbb N}}$ is a pseudo-Riesz basis, then
	$\{y_n\}_{n \in {\mathbb N}}$ is a pseudo-Riesz basis. Moreover, $\{ y_n \}_{n=k+1}^\infty$ 
	is a modular Riesz basis of the Hilbert $A$-submodule $\overline{\operatorname{Span}_{A} \left\{ y_n \right\}_{n=k+1}^{\infty}}$,
	$\overline{\operatorname{Span}_{A} \left\{ y_n \right\}_{n=k+1}^{\infty}}$ is orthogonally complemented in $l_2(A)$, and 
	$\overline{\operatorname{Span}_{A} \left\{ x_n \right\}_{n=k+1}^{\infty}}^{\perp}$ is finitely generated for any $k \geq N$.
	
	(3) If $\{x_n\}_{n \in {\mathbb N}}$ is a pseudo-Riesz-basis but
	not a pseudo-Riesz-Weyl basis, then
	$\{y_n\}_{n \in {\mathbb N}}$ is not a pseudo-Riesz-Weyl basis.		
\end{prop}

\begin{proof}
	(1) Let $\varphi$ be the synthesis operator of $\{y_n\}_{n=1}^{\infty}$. Again, by the same arguments as the proof of Proposition \ref{perturbacija}, we deduce that $\varphi$ has the matricial decomposition
\begin{equation*}
	\begin{pmatrix}
		\widetilde{\varphi}_{1} & 0 \\
		0 & \widetilde{\varphi}_{4}
	\end{pmatrix}
\end{equation*}
with respect to the decomposition
\begin{equation*}
	l_2(A) = L_{N}^{\perp} \dotplus U(L_{N} \overset{\varphi}{\longrightarrow} 	V\bigl(\theta^{*}(L_{N}^{\perp})\bigr) \dotplus 	\bigl(\theta^{*}(L_{N}^{\perp})\bigr)^{\perp} = 	l_2(A),
\end{equation*}
where $\widetilde{\varphi}_{1}$, $U$, and $V$ are isomorphisms and $\theta^{*}$ is the synthesis operator of $\{x_n\}_{n=1}^{\infty}$. Moreover, by the proof of Proposition \ref{perturbacija}, there exists an isomorphism $\widetilde{U}$ of $l_2(A) $ such that the decomposition
\begin{equation*}
	l_2(A) = 	L_{N}^{\perp} \dotplus \widetilde{U}(L_{N}) \overset{\theta^{*}}{\longrightarrow} \theta^{*}(L_{N}^{\perp}) \oplus \bigl(\theta^{*}(L_{N}^{\perp})\bigr)^{\perp} = l_2(A)
\end{equation*}
is an upper semi-$A$-Fredholm decomposition for $\theta^{*}$. If $\{x_n\}_{n \in {\mathbb N}}$ is not a pseudo-Riesz basis, then $\theta^{*}$ is not $A$-Fredholm, hence $\bigl(\theta^{*}(L_{N}^{\perp})\bigr)^{\perp}$ cannot be finitely generated. Now suppose that $\varphi $ is $A$-Fredholm. Then there would exist an $A$-Fredholm decomposition
\begin{equation*}
	l_2(A) = {\mathcal M}_{1}\dotplus {\mathcal N}_{1} \overset{\theta^{*}}{\longrightarrow} 	{\mathcal M}_{2}\dotplus {\mathcal N}_{2} = l_2(A)
\end{equation*}	
for $ \varphi$,  in particular ${\mathcal N}_{2} $ would be finitely generated. Since the decomposition
\begin{equation*}
	l_2(A) u = L_{N}^{\perp}\dotplus U(L_{N}) 	\overset{\varphi}{\longrightarrow} V\bigl(\theta^{*}(L_{N}^{\perp})\bigr) \dotplus \bigl(\theta^{*}(L_{N}^{\perp})\bigr)^{\perp} = 	l_2(A)
\end{equation*}
is an $A$-Fredholm decomposition for $\varphi$,  by \cite[Lemma 3.1]{Ivkovic_2023} it would follow that $\bigl(\theta^{*}(L_{N}^{\perp})\bigr)^{\perp}$	has to be finitely generated, which is a contradiction. Thus, $\varphi$ cannot be $A$-Fredholm. Hence, $\{y_n\}_{n \in {\mathbb N}}$ cannot be a pseudo-Riesz basis.
	
	(2) If $\{x_n\}_{n=1}^{\infty}$ is a pseudo-Riesz basis, then $\theta^{*}$ is $A$-Fredholm. Since	
\begin{equation*}
	l_2(A) = L_{N}^{\perp} \dotplus \widetilde{U}(L_{N})	\overset{\theta^{*}}{\longrightarrow} \theta^{*}(L_{N}^{\perp}) \oplus \bigl(\theta^{*}(L_{N}^{\perp})\bigr)^{\perp} = l_2(A)
\end{equation*}
is an upper $A$-Fredholm decomposition for $\theta^{*}$, by \cite[Lemma 3.1]{Ivkovic_2023} it follows that $\bigl(\theta^{*}(L_{N}^{\perp})\bigr)^{\perp}$ is finitely generated. Hence, since
\begin{equation*}
	l_2(A) = L_{N}^{\perp} \dotplus U(L_{N}) \overset{\varphi}{\longrightarrow} V\bigl(\theta^{*}(L_{N}^{\perp})\bigr) \dotplus \bigl(\theta^{*}(L_{N}^{\perp})\bigr)^{\perp} = l_2(A)
\end{equation*}
is an upper $A$-Fredholm decomposition for $\varphi$ and $\bigl(\theta^{*}(L_{N}^{\perp})\bigr)^{\perp}$ is finitely generated, it is actually an $A$-Fredholm decomposition for $\varphi$, so $\varphi$ is $A$-Fredholm. Thus $\{y_n\}_{n \in {\mathbb N}}$ is a pseudo-Riesz-basis. The analogous statement follows for any $k \geq N$ from Proposition \ref{pseudo-RW_char}, (1).
	
	(3) If $\{x_n\}_{n \in {\mathbb N}}$ is a pseudo-Riesz basis but not a pseudo-Riesz-Weyl basis, then $\theta^{*}$ is $A$-Fredholm, but not $A$-Weyl. Then the decompositions
\begin{equation*}
	l_2(A) = L_{N}^{\perp} \dotplus \widetilde{U}(L_{N}) 	\overset{\theta^{*}}{\longrightarrow} \theta^{*}(L_{N}^{\perp}) \oplus \bigl(\theta^{*}(L_{N}^{\perp})\bigr)^{\perp} = l_2(A)
\end{equation*}
and
\begin{equation*}
	l_2(A) = L_{N}^{\perp} \dotplus U(L_{N}) 	\overset{\varphi}{\longrightarrow} V\bigl(\theta^{*}(L_{N}^{\perp})\bigr) \dotplus \bigl(\theta^{*}(L_{N}^{\perp})\bigr)^{\perp} = 	l_2(A)
\end{equation*}
are $A$-Fredholm decompositions for $\theta^{*}$ and $\varphi$, respectively, because again by \cite[Lemma 3.1]{Ivkovic_2023} we must have that 	$	\bigl(\theta^{*}(L_{N}^{\perp})\bigr)^{\perp}$ 	is finitely generated in this case. 
	
Suppose now that $\varphi$ is $A$-Weyl. Since  $\widetilde{U}(L_{N}) \cong U(L_{N})$ and $V\bigl(\theta^{*}(L_{N}^{\perp})\bigr) 	\cong \theta^{*}(L_{N}^{\perp})$, then, again by  letting $ \theta^{*} $ and $\varphi$ play the role of operators $F$ and $D,$ respectively, in the proof of \cite[Theorem 4.2]{Ivkovic_2023}, we would be able to construct an $A$-Fredholm decomposition for $\theta^{*}$ with ${\mathcal N}_1 \cong {\mathcal N}_2$, which would be a contradiction since $\theta^{*}$ is assumed to be \underline{not} $A$-Weyl. Thus, we have that $\varphi$ is not $A$-Weyl, and hence, $\{y_n\}_{n \in {\mathbb N}}$ cannot be a pseudo-Riesz-Weyl basis. 
\end{proof}

\begin{prop} \label{prop_perturb3}
	Let $\{y_n\}_{n \in {\mathbb N}}$ be a Bessel sequence in $l_2(A)$ such that there exists a sequence $\{M_n\}_{n \in {\mathbb N}}$
	of positive numbers satisfying
\begin{equation*}
	\lim_{n \to \infty} M_n = 0
\end{equation*}
and, in addition,
\begin{equation*}
	\sum_{j=n}^{\infty} \langle x,y_j \rangle \langle y_j,x \rangle \leq M_n \langle x,x \rangle
\end{equation*}
for all $x \in l_2(A)$ and $n \in {\mathbb N}$. Then the following statements hold.
	
	(1) If $\{x_n\}_{n \in {\mathbb N}}$ is a pseudo-Riesz-sequence in $l_2(A)$,
	then $\{x_n+y_n\}_{n \in {\mathbb N}}$ is a pseudo-Riesz-sequence in $l_2(A)$.
	
	(2) If $\{x_n\}_{n \in {\mathbb N}}$ is a pseudo-Riesz-basis in $l_2(A)$, then
	$\{x_n+y_n\}_{n \in {\mathbb N}}$ is a pseudo-Riesz-basis in $l_2(A)$.
	
	(3) If $\{x_n\}_{n \in {\mathbb N}}$ is a pseudo-Riesz-Weyl-sequence in $l_2(A)$, then
	$\{x_n+y_n\}_{n \in {\mathbb N}}$ is a pseudo-Riesz-Weyl-sequence in $l_2(A)$.
	
	(4) If $\{x_n\}_{n \in {\mathbb N}}$ is a pseudo-Riesz-Weyl-basis in $l_2(A)$, then
	$\{x_n+y_n\}_{n \in {\mathbb N}}$ is a pseudo-Riesz-Weyl-basis 	in $l_2(A)$.
\end{prop}

\begin{proof}
	Notice that if $\varphi$ is the synthesis operator of $\{y_n\}_{n \in {\mathbb N}}$, and $Q_m$ denotes the orthogonal projection onto $L_m^{\perp}$ for each $m \in {\mathbb N}$, 	then $\varphi Q_m$ is the synthesis operator of the Bessel sequence
\begin{equation*}
	\{ \underbrace{0,\ldots,0}_{m\text{-times}}, g_{m+1},g_{m+2},\ldots 	\} \, ,
\end{equation*}
which by the assumption has (upper) bound 	$M_m$ for each $m \in {\mathbb N}$. By Proposition \ref{prop_sequ}, we then have
$\|\varphi Q_m\| \leq \sqrt{M_m}$ for each $m \in {\mathbb N}$, and hence, $\|\varphi Q_m\| \to 0$ as $m$ goes to infinity by assumption.  In other words, $
\varphi_{|_{L_m^{\perp}}} \to 0$ as $m$ goes to infinity. By \cite[Proposition 2.2.7]{MT} this means that $\varphi$ is compact, i.e. $	\varphi \in {\rm K}_A(l_2(A))$. Since the operator sets of all semi-$A$-Fredholm, all $A$-Fredholm, all semi-$A$-Weyl and all $A$-Weyl operators on $l_2(A)$, respectively, are invariant under ``compact'' perturbations by \cite[Lemma 2.7.13]{MT} and \cite[Lemma 5.9]{Ivkovic_2019}, the statements follow from Proposition \ref{semi-weyl} and Proposition \ref{weyl}.
\end{proof}

\section*{Acknowledgment}
The second author thanks D.~S.~Djordjevi\'c for fruitful discussions on various aspects of semi-Fredholm operator theory and on the ways to extend it to the context of Hilbert C*-modules. The authors are grateful to D.~Baki\'c and L.~Cigler for their feedback on an intermediate version of the manuscript.


\end{document}